\documentclass[11pt]{amsart}
\usepackage{amsmath}
\usepackage[latin1]{inputenc}
\newtheorem{thm}{Theorem}[section]
\newtheorem{cor}[thm]{Corollary}
\newtheorem{lem}[thm]{Lemma}
\newtheorem{prop}[thm]{Proposition}
\theoremstyle{definition}

\theoremstyle{remark}

\newtheorem{exe}[thm]{\bf Example}
\numberwithin{equation}{section}

\begin{document}
\title[Topological transitive abelian subgroups of GL($\mathbf{n}, \mathbb{R})$
\dag]{Topological transitive abelian subgroups of GL($\mathbf{n}, \mathbb{R})$ \dag}

\author{Adlene Ayadi, Habib Marzougui, Ezzeddine Salhi}

\address{Adlene Ayadi$^{1}$,  University of Gafsa, Faculty of Science of Gafsa, Department of Mathematics, Tunisia; Habib Marzougui$^{2}$, \
University  of  7$^{th}$  November  at  Carthage, Faculty of Science of Bizerte, Department of Mathematics, Zarzouna. 7021. Tunisia.;  Ezzeddine
Salhi$^{3}$,  University of Sfax, Faculty of Science of
Sfax, Department of Mathematics, \, B.P. 802, 3018 Sfax, Tunisia}

\email{$$ \begin{array}{l} adleneso@yahoo.fr; \ \
habib.marzouki@fsb.rnu.tn;\ \ Ezzeddine.Salhi@fss.rnu.tn.
\end{array}$$}

\thanks{$\dag$ This work is supported by the research unit: syst\`emes dynamiques et combinatoire:
99UR15-15 and it was done within the framework of the Associateship Scheme of the
Abdus Salam ICTP, Trieste, Italy. This paper is a new version with major revision of a previous one which
was circulated in ICTP preprint \cite{aAhMeS} in 2006.}


 \begin{abstract}
We give a complete characterization of abelian subgroups of
GL($n,\ \mathbb{R})$ with a locally dense (resp. dense) orbit in  $\mathbb{R}^{n}$.  For finitely generated subgroups, this
characterization is explicit and it is used to show that no abelian subgroup of
GL($n,\ \mathbb{R})$ generated by $[\frac{n+1}{2}]$ matrices can
have a dense orbit in  $\mathbb{R}^{n}$. $($[\ ] denotes the integer part$)$
 \end{abstract}
\maketitle
\bigskip

\section{\bf Introduction}
Let $M_{n}(\mathbb{R})$  be the vector space of square matrices over $\mathbb{R}$ of
order  $n\geq 1$, GL($n,
\mathbb{R})$  be the group of all invertible elements of
$M_{n}(\mathbb{R})$  and let  $G$  be an abelian subgroup of
GL($n, \mathbb{R})$.  If  $G$  has a finite set of generators, then it is said to be
\emph{finitely generated}. There is a linear natural action  GL($n,
\mathbb{R})\times \mathbb{R}^{n}\longrightarrow \mathbb{R}^{n }:
(A, v)\longmapsto Av$. For a vector  $v\in \mathbb{R}^{n}$,
denote by  $G(v) = \{Av:\  A\in G \}\subset \mathbb{\mathbb{R}}^{n}$
the \emph{orbit} of $G$ through  $v$. The orbit  $G(v)$  is
\textit{locally dense} in  $\mathbb{R}^{n}$  if the closure
$\overline{G(v)}$ has no empty interior. It is dense in
$\mathbb{R}^{n}$  if $\overline{G(v)} = \mathbb{R}^{n}$. In this
paper, we are concerned with the existence of a dense orbit for
the linear natural action of $G$  on $\mathbb{R}^{n}$ in which case, $G$
is said to be \textit{topological transitive.}

The authors gave in \cite{aAhM06} a characterization of abelian
subgroups of GL($n, \mathbb{C})$ that are topologically transitive.
This paper can be viewed as a continuation of that work for the
real case.

We give in Section 8, examples of abelian topological transitive subgroup of GL($2, \mathbb{R})$
(resp. GL$(4,\ \mathbb{R}))$. Examples of non abelian topological transitive subgroups of GL($2,
\mathbb{R})$ were constructed in
\cite{fDanS00}, \cite{J}, (see also \cite{msK04}): Dal'bo and Starkov gave in \cite{fDanS00}, an
example of an infinitely generated {\it non abelian} subgroups of the special linear group SL($2, \ \mathbb{R})$
(i.e. consisting of matrices having determinant 1)
 with a dense orbit in  $\mathbb{R}^{2}$.
However, there are no {\it abelian} finitely generated subgroup of  SL($2, \ \mathbb{R})$
which are topological transitive. (See Corollary ~\ref{C:8}).

Notice that if  $G$  is a subgroup of the orthogonal group
U$_{n}(\mathbb{R})$ (consisting of matrices preserving the usual
inner product)
then every orbit $G(v)$ of  $G$  is contained in a sphere and so,
G($v)$ is not dense in  $\mathbb{R}^{n}$.

To state our main results, we need to introduce the following
notations and definitions:

 A subset $E \subset \mathbb{R}^{n}$  is called \emph{$G$-invariant} if
$A(E)\subset E$  for any  $A\in G$; that is  $E$  is a union of
orbits. If  $U$  is an open  $G$-invariant set, the orbit  $G(v)\subset U$
is called \emph{minimal in  $U$}  if \ $\overline{G(v)}\cap
U = \overline{G(w)}\cap U$  for every  $w\in \overline{G(v)}\cap
U$.
\
\\
\\
Denote by
\\
\textbullet \ $\mathbb{K} =
\mathbb{R}$ \textrm{or}  $\mathbb{C}$
\
\\
\textbullet \ $\mathbb{K}^{*}= \mathbb{K}\backslash\{0\}$ and
$\mathbb{N}_{0}= \mathbb{N}\backslash\{0\}$.
\
\\
Let
$n\in\mathbb{N}_{0}$  be fixed. For each $m=1,2,\dots, n$,
denote by
\\
\\
\textbullet \; $\mathbb{T}_{m}(\mathbb{K})$ the set of matrices over $\mathbb{K}$ of the form
$$\begin{bmatrix}
  \mu &  &   & 0 \\
  a_{2,1} & \mu &   &  \\
  \vdots &  \ddots & \ddots &  \\
  a_{m,1} & \dots & a_{m,m-1} & \mu
\end{bmatrix} \qquad (1)$$
\\
\\
\textbullet \; $\mathbb{T}_{m}^{\ast}(\mathbb{K})$  the group of matrices of the form $(1)$ with
$\mu\neq 0$.
\\
\\
\textbullet \; $\mathbb{T}_{m}^{+}(\mathbb{R})$ the group of
matrices over $\mathbb{R}$ of the form $(1)$ with  $\mu > 0$.
\
\\
\\
For each $1\leq m\leq \frac{n}{2}$, denote by
\\
\\
$\bullet$ \; $\mathbb{B}_{m}(\mathbb{R})$ the set of matrices of
$M_{2m}(\mathbb{R })$ of the form
$$\begin{bmatrix}
  C &  &   & 0 \\
  C_{2,1} & C &   &   \\
  \vdots &  \ddots & \ddots & \\
  C_{m,1} & \dots & C_{m,m-1} & C
\end{bmatrix}: \  C, \ C_{i,j}\in \mathbb{S }, \ 2\leq i\leq m, 1\leq j\leq m-1 \ \qquad (2)$$
where  $\mathbb{S}$ is the set of matrices over $\mathbb{R}$ of the form $$\begin{bmatrix}
  \alpha & \beta \\
  -\beta & \alpha \\
  \end{bmatrix}.$$
\\
\\
$\bullet$ \; $\mathbb{B}^{*}_{m}(\mathbb{R}):= \mathbb{B}_{m}(\mathbb{R})\cap \textrm{GL}(2m, \mathbb{R})$, it is a subgroup of $\textrm{GL}(2m, \mathbb{R})$.
\
\\
\\
Let  $r, \ s\in \mathbb{N}$ and $$\eta =
 \begin{cases}
 (n_{1},\dots,n_{r};\  m_{1},\dots,m_{s}) & \textrm{ if } rs\neq  0, \\
(m_{1},\dots,m_{s}) & \textrm{ if } r= 0, \\
(n_{1},\dots,n_{r}) & \textrm{ if } s=0
\end{cases}$$ be a sequence of positive integers such that $$(n_{1}+\dots + n_{r}) +2(m_{1}+\dots + m_{s})=n. \qquad  (3)$$
In particular, $r+2s\leq n$.
\
\\
Write
\
\\

\textbullet \; \ $\mathcal{K}_{\eta,r,s}(\mathbb{R}): = \mathbb{T}_{n_{1}}(\mathbb{R})\oplus\dots \oplus \mathbb{T}_{n_{r}}(\mathbb{R})\oplus
\mathbb{B}_{m_{1}}(\mathbb{R })\oplus\dots \oplus \mathbb{B}_{m_{s}}(\mathbb{R}).$
In particular:
\
\\
- If  $r=1, \ s=0$ then $\mathcal{K}_{\eta,1,0}(\mathbb{R}) = \mathbb{T}_{n}(\mathbb{R})$  and  $\eta=(n)$.
\
\\
- If $r=0, \ s=1$  then $\mathcal{K}_{\eta,0,1}(\mathbb{R}) =
\mathbb{B}_{m}(\mathbb{R})$  and  $\eta=(m)$,  $n=2m$.
\
\\
- If $r=0, \ s>1$  then $\mathcal{K}_{\eta,0,s}(\mathbb{R}) =
\mathbb{B}_{m_{1}}(\mathbb{R})\oplus\dots\oplus\mathbb{B}_{m_{s}}(\mathbb{R})$ and $\eta=(m_{1},\dots,m_{s})$.
\
\\
\textbullet \; $\mathcal{K}^{*}_{\eta,r,s}(\mathbb{R}): =
\mathcal{K}_{\eta,r,s}(\mathbb{R})\cap \textrm{GL}(n, \ \mathbb{R})$.
\\
\\
\textbullet \;  $\mathcal{K}^{+}_{\eta,r,s}(\mathbb{R}): = \mathbb{T}^{+}_{n_{1}}(\mathbb{R})\oplus\dots \oplus \mathbb{T}^{+}_{n_{r}}(\mathbb{R})\oplus
\mathbb{B}^{*}_{m_{1}}(\mathbb{R })\oplus\dots \oplus \mathbb{B}^{*}_{m_{s}}(\mathbb{R}).$
\\
\\
For a vector  $v\in \mathbb{C}^{n}$,  write\\
\textbullet \; Re($v$), Im($v)\in \mathbb{R}^{n}$ so that  $v = \textrm{Re}(v) +
i \textrm{Im}(v)$.
\\
\textbullet \; $\textrm{exp} :\ \mathbb{M}_{n}(\mathbb{R})
\longrightarrow\textrm{GL}(n, \mathbb{R})$  is the matrix
exponential map; set $\textrm{exp}(M) = e^{M}$.\\

Given any abelian subgroup  $G\subset \textrm{GL}(n, \mathbb{R})$,
there always exists a $P\in \textrm{GL}(n, \mathbb{R})$ and a
partition $\eta$ of $n$ such that $\widetilde{G}=P^{-1}GP\subset
\mathcal{K}^{*}_{\eta,r,s}(\mathbb{R})$. (See Proposition ~\ref{p:1}).
For such a choice of matrix $P$,  we let
\\
 - $\mathrm{g} = \textrm{exp}^{-1}(G)\cap \left[
P(\mathcal{K}_{\eta,r,s}(\mathbb{R}))P^{-1}\right]$. If $G\subset \mathcal{K}^{*}_{\eta,r,s}(\mathbb{R})$, then we have $\mathrm{g} =
\textrm{exp}^{-1}(G)\cap \mathcal{K}_{\eta,r,s}(\mathbb{R})$.
\
\\
- $\mathrm{g}_{u} = \{Bu: \ B\in \mathrm{g}\}, \ u\in\mathbb{R}^{n}.$
\
\\
- $G^{+}:= G\cap\mathcal{K}^{+}_{\eta,r,s}(\mathbb{R})$, if $r\geq 1$ and $G^{+}=G$ if $r=0$, it is a subgroup of $G$.
\
\\
- $G^{2} =\{A^{2}: {A\in G}\}$
\
\\
- $\mathrm{g}^{2} = \textrm{exp}^{-1}(G^{2})\cap \left[
P(\mathcal{K}_{\eta,r,s}(\mathbb{R}))P^{-1}\right]$
\
\\
\\
\textbullet \; Let $M\in G$, one can write $ \widetilde{M}:= P^{-1}M P=\mathrm{diag}(M_{1},\dots,M_{r}$; \ $\widetilde{M}_{1},\dots,\widetilde{M}_{s})\in
\mathcal{K}^{*}_{\eta,r,s}(\mathbb{R}).$ Let $\mu_{k}$ be the eigenvalue of $M_{k}$, $k=1,\dots,r$, and define the {\it index}
ind$(\widetilde{G})$ of $\widetilde{G}$ to be
\
\\
 $$\textrm{ind}(\widetilde{G}):= \begin{cases}

0, \  if \ \ r=0 \\
\\
\left\{\begin{array}{c}
1, \ \ \mathrm{if} \ \ \ \mathrm{\exists} \widetilde{M}\in \widetilde{G}\ \ \mathrm{with} \ \mu_{1}<0 \\
0, \  \ \mathrm{otherwise}\ \ \ \ \ \ \ \ \ \ \ \ \ \ \ \ \ \ \ \ \ \
\end{array}
\right., \ if \ \ r=1\\
\\
 \mathrm{card}\left\{k\in\{1,\dots,r\}: \ \exists \widetilde{M}\in \widetilde{G}, \ \mathrm{with} \ \mu_{k} < 0,
  \ \mu_{i} >0, \ \forall\  i\neq k\right\}, \ if \ \ r\notin\{0,\ 1\}.
 \end{cases}$$
 ($\mathrm{card}(E)$ denotes the number of elements of a subset $E$ of $\mathbb{N}$).\
 \\
In particular,
\\
- If  $\widetilde{G}\subset \mathcal{K}^{+}_{\eta,r,s}(\mathbb{R})$, then  $\textrm{ind}(\widetilde{G})=0\neq r$.
\\
- If $\widetilde{G}\subset \mathbb{B}^{*}_{m}(\mathbb{R})$, then ind($\widetilde{G})=0$ (since $r=0$).

 We define the {\it index} of $G$ to be  $\textrm{ind}(G): =\textrm{ind}(\widetilde{G})$. Obviously, this definition does not depend on $P$.
\\
\\
Denote by
\
\\
\textbullet \;  $\mathcal{B}_{0} = (e_{1},\dots,e_{n})$ the canonical
basis of $\mathbb{K}^{n}$  and by  $I_{n}$  the identity matrix.
\\
\textbullet \; $u_{0} = [e_{1,1},\dots,e_{r,1}; f_{1,1},\dots,
f_{s,1}]^{T}\in \mathbb{R}^{n}$ where for $k=1,\dots, r, \ l=1,\dots, s$,
 $$e_{k,1} = [1,0,\dots,
0]^{T}\in \mathbb{R}^{n_{k}}$$
and
$$f_{l,1} = [1,0,\dots, 0]^{T}\in
\mathbb{R}^{2m_{l}}.$$
\\
\textbullet \;   $v_{0} = Pu_{0}$.
\
\\
\\
\\
  \textbullet \; $f^{(l)} = [
0,\dots,0,f^{(l)}_{1},\dots,
  f^{(l)}_{s}]^{T}\in \mathbb{R}^{n}$

where for  $i=1,\dots, r; \ j=1,\dots, s$:
$$f^{(l)}_{j} =
  \begin{cases}
    0\in \mathbb{R}^{2m_{j}} & \mathrm{if}\ \ j\neq l, \\
       [0,1,0,\dots,0]^{T}\in
\mathbb{R}^{2m_{l}} & \mathrm{if}\ \ j=l .
  \end{cases}$$
\
\\
An equivalent formulation is $f^{(1)}=e_{t_{1}}$, $\dots, f^{(l)}=e_{t_{l}}$
where $t_{1}=\underset{j=1}{\overset{r}{\sum}}n_{j}+2,$ \ $t_{l}= \underset{j=1}{\overset{r}{\sum}}n_{j}+
2\underset{j=1}{\overset{l-1}{\sum}}m_{j}+2,$ \ \
 $l=2,\dots, s$.
\bigskip

For a \textit{finitely generated} subgroup $G\subset \textrm{GL}(n,
\mathbb{R})$, let introduce the following property. Consider the following rank condition on a collection of
matrices
 $A_{1},\dots,A_{p}\in \mathcal{K}_{\eta,r,s}(\mathbb{R})$:  We
say that  $A_{1},\dots,A_{p}$  satisfy the \emph{density property}
 if there exist  $B_{1},\dots,B_{p}\in
\mathcal{K}_{\eta,r,s}(\mathbb{R})$  such that  $A_{1}^{2} =
e^{B_{1}},\dots,A_{p}^{2} = e^{B_{p}}$  and for every
$(s_{1},\dots,s_{p}; t_{1},\dots,t_{s})\in
\mathbb{Z}^{p+s}\backslash \{0\}$:
$$\mathrm{rank}\left(\left[\begin{array}{cccccc}
  B_{1}u_{0}, & \dots, & B_{p}u_{0} &  2\pi f^{(1)}, & \dots & 2\pi f^{(s)} \\
  s_{1}, & \dots, & s_{p} & t_{1}, & \dots & t_{s}
\end{array}\right]\right) = n+1$$
\medskip
\bigskip

Our principal results can now be stated as follows: \medskip

\begin{thm} \label{T:1}Let  $G$  be an abelian subgroup of  GL$(n,
\ \mathbb{R})$. The following properties are equivalent:
\begin{enumerate}
    \item [(i)]  $G$  has a locally dense orbit in  $\mathbb{R}^{n}$
    \item [(ii)]  The orbit  $G(v_{0})$  is locally dense in  $\mathbb{R}^{n}$
    \item [(iii)]  $\mathrm{g}_{v_{0}}$  is an additive subgroup dense in  $\mathbb{R}^{n}$
\end{enumerate}
\end{thm}
\medskip

\begin{cor} \label{C:1} Let  $G$  be an abelian subgroup of  GL($n, \mathbb{R})$ and $P\in \textrm{GL}(n, \mathbb{R})$ so that
 $P^{-1}GP\subset\mathcal{K}^{*}_{\eta,r,s}(\mathbb{R})$, for some $1\leq r,s\leq n$. The following properties are equivalent:
 \begin{itemize}
   \item [(i)] $G$ is topological transitive.
   \item [(ii)] G($v_{0})$ is dense in $\mathbb{R}^{n}$.
   \item [(iii)]  $\mathrm{g}_{v_{0}}$  is an additive subgroup dense in
$\mathbb{R}^{n}$ and  $\textrm{ind}(G)=r$.
 \end{itemize}
\end{cor}
\medskip

\begin{cor}\label{C:2}
\begin{itemize}
\item [(i)]Let $G$  be an abelian subgroup of
$\mathbb{B}_{n}^{*}(\mathbb{R})$. If  $\mathrm{g}_{e_{1}}$ is
an additive subgroup dense in  $\mathbb{R}^{2n}$  then  $G$ is
topological transitive.
\item [(ii)] Let $G$  be an abelian subgroup of $\mathbb{T}_{n}^{*}(\mathbb{R})$. If  $\mathrm{g}_{e_{1}}$ is
an additive subgroup dense in  $\mathbb{R}^{n}$ and ind($G)=1$ then  $G$ is
topological transitive.
\end{itemize}
 \end{cor}
\
\\
For finitely generated abelian subgroups $G$ of GL($n, \mathbb{R})$, we have the following theorem:
\begin{thm}\label{T:2} Let  $G$  be an abelian subgroup of GL($n, \mathbb{R})$ and $P\in \textrm{GL}(n, \mathbb{R})$ so that
 $P^{-1}GP\subset\mathcal{K}^{*}_{\eta,r,s}(\mathbb{R})$, for some $1\leq r,s\leq n$. Let $A_{1},\dots,A_{p}$ generate $G$ and let  $B_{1},\dots,B_{p}\in \mathrm{g}$
such that $A_{1}^{2}= e^{B_{1}},\dots, A_{p}^{2} = e^{B_{p}}$.
The following properties are equivalent:

\begin{enumerate}
\item [(i)] $G$ has a locally dense orbit in $\mathbb{R}^{n}$.
\item [(ii)] $G(v_{0})$ is locally dense in $\mathbb{R}^{n}$.
\item [(iii)]  $\mathrm{g}^{2}_{v_{0}}=\underset{k=1}{\overset{p}{\sum}}\mathbb{Z}(B_{k}v_{0}) +
\underset{l=1}{\overset{s}{\sum}}2\pi\mathbb{Z}Pf^{(l)}$ is dense in $\mathbb{R}^{n}$.
\end{enumerate}
\end{thm}
\medskip

\begin{cor}\label{CC:2} Under the hypothesis of Theorem 1.4, the following properties are equivalent:

\begin{enumerate}
\item [(i)]  $G$  is topological transitive.
\item [(ii)]  $P^{-1}A_{1}P,\dots,P^{-1}A_{p}P$  satisfy the \emph{density property} and  $\textrm{ind}(G) = r$.
\item [(iii)] $\mathrm{g}^{2}_{v_{0}}=\underset{k=1}{\overset{p}{\sum}}\mathbb{Z}(B_{k}v_{0}) +
\underset{l=1}{\overset{s}{\sum}}2\pi\mathbb{Z}Pf^{(l)}$ is dense in $\mathbb{R}^{n}$ and  $\textrm{ind}(G) = r$.
\end{enumerate}
\end{cor}
\medskip

\begin{cor}\label{C:3} Let  $G$  be an abelian subgroup of GL($n, \mathbb{R})$ and $P\in \textrm{GL}(n, \mathbb{R})$ so that
$P^{-1}GP\subset\mathcal{K}^{*}_{\eta,r,s}(\mathbb{R})$. If $G$ is generated by $p$ matrices with  $p\leq n-s$,
then it has nowhere dense orbit. In particular, $G$ is not topological transitive.
\end{cor}
\smallskip

\begin{cor}\label{C:4} Let  $G$  be an abelian subgroup of GL($n, \mathbb{R})$: If $G$ is generated by $p$ matrices with
 $p\leq [\frac{n+1}{2}]$, then it has nowhere dense orbit. In particular, $G$ is not topological transitive. $( [\ ]$ denotes the integer part.$)$
\end{cor}
\
\\
\\
\textbf{Remark 1}. Corollaries ~\ref{C:3}  and ~\ref{C:4} are not true in general if  $p> n-s$  (resp.  $p >
[\frac{n+1}{2}]$) as can be shown in Example 8.4.
\bigskip

This paper is organized as follows: In Section $2$,  we introduce
the triangular representation for an abelian subgroup of $GL(n, \mathbb{R})$. In Section $3$,  we give some basic properties
 of the matrix exponential map and the additive group $\mathrm{g}$ associated to the group $G$.
 Section 4 is devoted to some properties related to subgroups of $\mathcal{K}^{*}_{\eta,r,s}(\mathbb{R})$ with a dense orbit.
 A parametrization  of an abelian subgroup of $\mathbb{T}^{*}_{n}(\mathbb{R})$ and some related properties are given in Section 5.
 Section 6 (resp. Section 7) gives some properties of the abelian subgroups of $\mathbb{T}^{*}_{n}(\mathbb{K})$
  (resp. $\mathcal{K}^{*}_{\eta,r,s}(\mathbb{R})$) with locally dense orbit.
  The proof of Theorems ~\ref{T:1} and ~\ref{T:2}, Corollaries ~\ref{C:1}, ~\ref{C:2}, ~\ref{C:3} and ~\ref{C:4} are done in Section  $8$.
 Section 9 is devoted to the special case $n=2$ and some examples. In Section 10 , we have included as an appendix, a detailed proof
  of some results in Section 3 
 seem rather difficult to find in the literature.
\medskip

\section{\bf Normal form of  abelian  subgroups  of GL(n,
$\mathbb{R}$)}

In this section we introduce the triangular representation for an
abelian subgroup  $G\subset \textrm{GL}(n,\ \mathbb{R})$. As noted in the
introduction, this reduces the existence of a dense orbit to a
question concerning subgroups of
$\mathcal{K}^{*}_{\eta,r,s}(\mathbb{R})$.

\begin{lem}\label{L:01} Let G be an abelian subgroup of $GL(n,\mathbb{R})$. Then there exists a direct sum decomposition
$$\mathbb{R}^{n} = \underset{k=1}{\overset{r}{\bigoplus}}E_{k}\oplus \underset{l=1}{\overset{q}{\bigoplus}}F_{l}\ \ \ \ \ (1)$$
for some $r, \ q$, $0\leq r\leq n$, $0\leq q\leq \frac{n}{2}$, where $E_{k}$ $($resp. $F_{l})$ is a $G$-invariant
vector subspace of $\mathbb{R}^{n}$ of
dimension $n_{k}$ $($resp. 2$m_{l})$, $1\leq k\leq r$ $($resp. $1\leq l\leq q)$, such that, for each $A\in G$ the
 restriction $A_{k}$ $($resp.
$\widetilde{A}_{l})$ of $A$ to $E_{k}$ $($resp. $F_{l}$ $)$ has a unique real eigenvalue $\lambda_{A,k}$ $($resp.
two conjugates complex eigenvalues $\mu_{A,l}$ and $\overline{\mu}_{A,l}$).
\end{lem}
\medskip

\begin{proof}Given $A\in G$, let $\lambda_{A,k}$ $($resp. $\mu_{A,l}$ and $\overline{\mu}_{A,l}$) be a real eigenvalue $($resp.
two nonreal conjugates complex eigenvalues$)$ and
$E_{A,k} = \textrm{Ker}(A - \lambda_{A,k}I_{n})^{n_{k}}$ (resp. $F_{A,l} = \textrm{Ker}\left((A - \mu_{A,l}I_{n})(A - \overline{\mu_{A,l}}I_{n})\right)^{m_{l}}$) the associated
generalized eigenspace. For any $B\in G$, the space
$E_{A,k}$ (resp. $F_{A,l}$) is invariant under $B$. If $B$ restricted to $E_{A,k}$ (resp. $F_{A,l}$) has
two distinct real eigenvalues (resp. two non conjugates complex eigenvalues), then it can be decomposed further. The decomposition (1) is the
maximal decomposition associated to all $A\in G$.
\end{proof}
\
\\
The restriction of the group $G$ to each subspace $E_{k}$ (resp. $F_{l}$) can be put into
triangular form (resp. in $\mathbb{B}_{m_{l}}(\mathbb{R})$'s form). This follows from the Lemmas 2.2 and 2.3 below.
\
\\
\begin{lem}\label{L:33}  Let $G$  be an abelian subgroup of GL($n, \mathbb{K})$. Assume that every element of $G$ has a unique
 eigenvalue. Then there exists a matrix $Q\in \textrm{GL}(n, \mathbb{K})$  such that  $Q^{-1}GQ$ is a subgroup of
 $\mathbb{T}^{*}_{n}(\mathbb{K})$.
\end{lem}
\medskip

\begin{proof}The proof is done by induction on $n\geq 1$.

For $n=1$, the Lemma is obvious. Suppose the Lemma is true for $n-1$, $n\geq 2$, and let $G$ be an abelian group of matrices in
GL($n, \mathbb{K})$ having only one real eigenvalue. Then there exists a common eigenvector $u\in\mathbb{K}^{n}$ for all matrices of $G$. Let
 $u_{1},\dots,u_{n-1}\in\mathbb{K}^{n}$ so that $\mathcal{B}:=(u_{1},\dots,u_{n-1},u)$ is a basis of $\mathbb{K}^{n}$.
  Let $P$ be the matrix of basis change from $\mathcal{B}_{0}$ to $\mathcal{B}$. Then, for every $A\in G$, we have
$$P^{-1}AP=\left[\begin{array}{cc}
A_{1} & 0 \\
L_{A} & \lambda_{A}
\end{array}
  \right],$$
where \ $A_{1}\in \textrm{GL}(n-1,\mathbb{K})$, \ $L_{A}= (a_{n,1},\dots,a_{n,n-1})\in M_{1,n-1}(\mathbb{K}).$

Denote by $G_{1}=\{A_{1}: \ A\in G\}$. One can check that $G_{1}$ is an abelian group of matrices in $\mathrm{GL}(n-1, \mathbb{K})$
having only one eigenvalue. By induction hypothesis, there exists $Q_{1}\in \textrm{GL}(n-1, \mathbb{K})$  such that
$Q_{1}^{-1}G_{1}Q_{1}$
is a subgroup of  $\mathbb{T}^{*}_{n-1}(\mathbb{K})$. Set $Q^{\prime} = \mathrm{diag}(Q_{1},1)$ and
$Q:=PQ^{\prime}\in \textrm{GL}(n, \mathbb{K})$. Then $Q\in \textrm{GL}(n, \mathbb{K})$ and for every $A\in G$, we have
 $$Q^{-1}AQ = (Q^{\prime})^{-1}(P^{-1}AP)Q^{\prime} = \left[\begin{array}{cc}
 A^{\prime}_{1} & 0 \\
 L^{\prime} & \lambda_{A}
 \end{array}
  \right],$$
where $$A^{\prime}_{1}:= Q_{1}^{-1}A_{1}Q_{1}= \begin{bmatrix}
  \lambda_{A} &  &   & 0 \\
  c_{2,1} & \lambda_{A} &   &  \\
  \vdots &  \ddots & \ddots &  \\
  c_{n,1} & \dots & c_{n,n-1} & \lambda_{A}
\end{bmatrix}\in \mathbb{T}^{*}_{n-1}(\mathbb{K})$$ and $L^{\prime}:= L_{A}Q_{1}
\in M_{1,n-1}(\mathbb{K})$. Therefore  and $Q^{-1}GQ$ is an abelian subgroup of
$\mathbb{T}^{*}_{n}(\mathbb{K})$. This completes the proof.
\end{proof}
\
\\
Let consider the following basis change:
\
\\
Assume that  $n=2m$, $m\in\mathbb{N}_{0}$.  For every  $k=1,\dots,m$, we let:  \\ $u_{k} =
\frac{e_{2k-1}-ie_{2k}}{2}$ and $\mathcal{C}_{0} =
(u_{1},\dots,u_{m}, \ \overline{u_{1}},\dots,\overline{u_{m}}),$   where $\overline{u}=(\overline{z_{1}},\dots,\overline{z_{m}})$
is the conjugate of $u=(z_{1},\dots, z_{m})$. Then $\mathcal{C}_{0}$ is a
basis of  $\mathbb{C}^{2m}$. Denote by $Q\in \textrm{GL}(2m, \mathbb{C})$  the matrix of basis change from
$\mathcal{B}_{0}$  to  $\mathcal{C}_{0}$.
\medskip

\begin{lem}\label{L:1} Under the notation above, for every $B\in\mathbb{B}_{m}(\mathbb{R})$,
$Q^{-1}BQ = \mathrm{diag}(B^{\prime}_{1},\
\overline{B^{\prime}_{1}} )$  where $B^{\prime}_{1}\in
\mathbb{T}_{m}(\mathbb{C})$.
\end{lem}
\medskip

\begin{proof} Let $B\in\mathbb{B}_{m}(\mathbb{R})$. Then  $B$ has the form: $$B=\left[\begin{array}{cccc}
                               C  & \ & \ & 0 \\
                               C_{2,1} & \ddots & \ & \ \\
                               \vdots & \ddots & \ddots  & \  \\
                               C_{m,1} & \dots & C_{m,m-1} & C
                             \end{array}
\right]$$ where $C=\left[\begin{array}{cc}
                           \alpha & \beta \\
                           -\beta & \alpha
                         \end{array}
\right]$ and $C_{i,j}=\left[\begin{array}{cc}
                           \alpha_{i,j} & \beta_{i,j} \\
                           -\beta_{i,j} & \alpha_{i,j}
                         \end{array}
\right]$; $1\leq j<i\leq m $. For every $1\leq j \leq m$, we have $$\left\{\begin{array}{c}
                                               Be_{2j-1}=\alpha e_{2j-1}-\beta e_{2j} +\underset{k=j+1}{\overset{m}{\sum}}\left(\alpha_{k,j} e_{2k-1}-\beta_{k,j} e_{2k}\right) \\
                                               Be_{2j}=\beta e_{2j-1}+\alpha e_{2j} +\underset{k=j+1}{\overset{m}{\sum}}\left(\beta_{k,j} e_{2k-1}+\alpha_{k,j} e_{2k}\right)\ \ \ \
                                             \end{array}\right.$$
So \begin{align*}
Bu_{j} & =B\left(\frac{e_{2j-1}-ie_{2j}}{2}\right)\\
\ & =\frac{1}{2}\left((\alpha -i\beta)e_{2j-1}-(\beta +i\alpha) e_{2j}\right) +\frac{1}{2}\underset{k=j+1}{\overset{m}{\sum}}\left((\alpha_{k,j} -i\beta_{k,j})e_{2k-1}-(\beta_{k,j} +i\alpha_{k,j}) e_{2k}\right)\\
\ &  =\left(\frac{\alpha -i\beta}{2}\right)\left(\frac{e_{2j-1}-ie_{2j}}{2}\right) +\underset{k=j+1}{\overset{m}{\sum}}\left(\frac{\alpha_{k,j} -i\beta_{k,j}}{2}\right)\left(\frac{e_{2k-1}-ie_{2k}}{2}\right)
\end{align*}
\
\\
 Write   $\lambda=\frac{\alpha -i\beta}{2}$  and  $\lambda_{k,j}=\frac{\alpha_{k,j} -i\beta_{k,j}}{2}$.  It follows that for every  $1\leq j \leq m-1$,
$$Bu_{j}=\lambda u_{j} +\underset{k=j+1}{\overset{m}{\sum}}\lambda_{k,j}u_{k}$$ \ \ \ \ and
$$B\overline{u_{j}}=\overline{\lambda} \overline{u_{j}} +\underset{k=j+1}{\overset{m}{\sum}}\overline{\lambda_{k,j}}\overline{u_{k}}.$$

 In the basis \ $\mathcal{C}_{0}=(u_{1},\dots,u_{m};\ \overline{u_{1}},\dots,\overline{u_{m}})$ \ of \ $\mathbb{C}^{2m}$, \ $Q^{-1}BQ=\mathrm{diag}(B^{\prime}_{1},\
\overline{B^{\prime}_{1}} )$  where $B^{\prime}_{1}=\left[\begin{array}{cccc}
                                                            \lambda & \ & \ & 0 \\
                                                            \lambda_{2,1} & \ddots & \ & \ \\
                                                            \vdots & \ddots & \ddots & \ \\
                                                            \lambda_{m,1}& \dots & \lambda_{m,m-1} & \lambda
                                                          \end{array}\right]\in
\mathbb{T}_{m}(\mathbb{C})$.
 \end{proof}
\
\\
\begin{lem}\label{r:01} Let $G$ be an abelian subgroup of GL($2m, \mathbb{C})$ and let $\mathcal{C}_{0}:=(v_{1},\dots,v_{m};\ \overline{v_{1}},\dots,\overline{v_{m}})$
 be a basis of $\mathbb{C}^{2m}$. Then:
\begin{itemize}
  \item [(i)]  $\mathcal{C}: =(\textrm{Re}(v_{1}), \textrm{Im}(v_{1}),\dots,\textrm{Re}(v_{m}), \textrm{Im}(v_{m}))$ is a basis of $\mathbb{R}^{2m}$.
  \item [(ii)] Set $P_{0}$ $($resp. $P)$ be the matrix of basis change from $\mathcal{B}_{0}$ to
$\mathcal{C}_{0}$ $($resp. $\mathcal{C})$. If for every $B\in G$,
$P_{0}^{-1}BP_{0}= \mathrm{diag}(B_{1},\overline{B_{1}})$ with
$B_{1}\in\mathbb{T}_{m}(\mathbb{C})$ then $P^{-1}BP\in \mathbb{B}_{m}(\mathbb{R})$.
\end{itemize}
\end{lem}

\begin{proof}(i) Let $\alpha_{1},\dots,\alpha_{m},\beta_{1},\dots,\beta_{m}\in\mathbb{R}$ such that
 $$\underset{k=1}{\overset{m}{\sum}}\alpha_{k} Re(v_{k})+\underset{k=1}{\overset{m}{\sum}}\beta_{k} \textrm{Im}(v_{k})=0.$$
So
\begin{align*}
 0 & \ =\underset{k=1}{\overset{m}{\sum}}\alpha_{k} \frac{v_{k}+\overline{v_{k}}}{2}-i\underset{k=1}{\overset{m}{\sum}}\beta_{k}\frac{v_{k}-\overline{v_{k}}}{2}\\
 \ & \ =\underset{k=1}{\overset{m}{\sum}}\frac{\alpha_{k}-i\beta_{k}}{2} v_{k}+\underset{k=1}{\overset{m}{\sum}}\frac{\alpha_{k}+i\beta_{k}}{2}\overline{v_{k}}
\end{align*}\
\\
Since $\mathcal{C}_{0}$ is a basis of $\mathbb{C}^{2m}$, $\alpha_{k}-i\beta_{k}=\alpha_{k}+i\beta_{k}=0$ for every $1\leq k\leq m$.
So $\alpha_{k}=\beta_{k}=0$, for every $1\leq k\leq m$, this proves that $\mathcal{C}$ is a basis of
$\mathbb{R}^{2m}$.
\
\\
\\
(ii) Write $B^{\prime}= P^{-1}BP =\textrm{diag}(B_{1},\overline{B_{1}})\in G^{\prime}$ where $$B_{1}=\left[\begin{array}{cccc}
                                                            \lambda & \ & \ & 0 \\
                                                            \lambda_{2,1} & \ddots & \ & \ \\
                                                            \vdots & \ddots & \ddots & \ \\
                                                            \lambda_{m,1}& \dots & \lambda_{m,m-1} & \lambda
                                                          \end{array}
\right]\in\mathbb{T}_{m}(\mathbb{C}).$$ It follows that for every  $j=1,\dots, m-1$, $$Bv_{j}=\lambda v_{j} +\underset{k=j+1}{\overset{m}{\sum}}\lambda_{k,j}v_{k} \ \ \ \ \mathrm{and} \ \ \  \ \
B\overline{v_{j}}=\overline{\lambda} \overline{v_{j}} +\underset{k=j+1}{\overset{m}{\sum}}\overline{\lambda_{k,j}}\overline{v_{k}}.$$
Write  $v_{k}=a_{k}+ib_{k}$,  $a_{k},b_{k}\in\mathbb{R}^{n}$, $\lambda=\alpha -i\beta$, $\lambda_{k,j}=\alpha_{k,j} +i\beta_{k,j}$
with $\alpha,\beta,\alpha_{k,j},\beta_{k,j}\in \mathbb{R}$.
So $a_{k}=\frac{v_{k}+\overline{v_{k}}}{2}$ and $b_{k}=\frac{v_{k}-\overline{v_{k}}}{2i}$. Then

\begin{align*}
Ba_{j} & =B\left(\frac{v_{j}+\overline{v_{j}}}{2}\right)\\
\ & = \frac{1}{2}\left(\lambda v_{j} +\underset{k=j+1}{\overset{m}{\sum}}\lambda_{k,j}v_{k}\right)+\frac{1}{2}\left(\overline{\lambda} \overline{v_{j}} +\underset{k=j+1}{\overset{m}{\sum}}\overline{\lambda_{k,j}}\overline{v_{k}}\right)\\
\ & = \frac{1}{2}\left((\lambda v_{j}+\overline{\lambda v_{j}}) +\underset{k=j+1}{\overset{m}{\sum}}(\lambda_{k,j}v_{k}+\overline{\lambda_{k,j}v_{k}})\right)\\
\ & =(\alpha a_{j}-\beta b_{j}) +\underset{k=j+1}{\overset{m}{\sum}}(\alpha_{k,j}a_{k} -\beta_{k,j}b_{k})\\
\end{align*}
and

\begin{align*}
Bb_{j} & = B\left(\frac{v_{j}-\overline{v_{j}}}{2i}\right)\\
\ & = \frac{1}{2i}\left(\lambda v_{j} +\underset{k=j+1}{\overset{m}{\sum}}\lambda_{k,j}v_{k}\right)-\frac{1}{2i}\left(\overline{\lambda} \overline{v_{j}} +\underset{k=j+1}{\overset{m}{\sum}}\overline{\lambda_{k,j}}\overline{v_{k}}\right)\\
\ & = \frac{1}{2i}\left((\lambda v_{j}-\overline{\lambda v_{j}}) +\underset{k=j+1}{\overset{m}{\sum}}(\lambda_{k,j}v_{k}-\overline{\lambda_{k,j}v_{k}})\right)\\
\ & =(\beta a_{j}+\alpha b_{j}) +\underset{k=j+1}{\overset{m}{\sum}}(\alpha_{k,j}b_{k} +\beta_{k,j}a_{k})\\
\end{align*}

Therefore   $$P^{-1}BP= \left[\begin{array}{cccc}
                               C  & \ & \ & 0 \\
                               C_{2,1} & \ddots & \ & \ \\
                               \vdots & \ddots & \ddots  & \  \\
                               C_{m,1} & \dots & C_{m,m-1} & C
                             \end{array}
\right]$$ where $C= \left[\begin{array}{cc}
                           \alpha & \beta \\
                           -\beta & \alpha
                         \end{array}
\right]$ and $C_{i,j}= \left[\begin{array}{cc}
                           \alpha_{i,j} & \beta_{i,j} \\
                           -\beta_{i,j} & \alpha_{i,j}
                         \end{array}
\right]$; $1\leq j<i\leq m$. Hence $P^{-1}BP\in\mathbb{B}_{m}(\mathbb{R})$.
\end{proof}
\bigskip

\begin{lem}\label{L:333+} Let $G$ be an abelian subgroup of
GL($2m,\mathbb{R}).$ Assume that every element of $G$ has two conjugates complex eigenvalues with one element $A\in G$
having two nonreal conjugates complex eigenvalues. Then there exists a $P\in \textrm{GL}(2m, \mathbb{R})$ such that
$P^{-1}GP\subset\mathbb{B}^{*}_{m_{1}}(\mathbb{R})\oplus\dots\oplus\mathbb{B}^{*}_{m_{s}}(\mathbb{R})$,
for some $1\leq s\leq m$ and $m_{1},\dots, m_{s}\in\mathbb{N}_{0}$ where $m_{1}+\dots +m_{s}=m$.
\end{lem}
\medskip

\begin{proof} We considered $G$ as an abelian subgroup of GL($2m, \mathbb{C})$.
Let $\lambda, \ \overline{\lambda}$ be two nonreal conjugates complex eigenvalues of $A$.
Let $F: = \textrm{Ker}(A-\lambda I_{2m})^{m}$ denote the generalized eigenspace of $A$ associated to
$\lambda$, so $\overline{F}: = \textrm{Ker}(A-\overline{\lambda} I_{2m})^{m}$ is the generalized eigenspace
 of $A$ associated to $\overline{\lambda}$ and we have
$\mathbb{C}^{2m}= F\oplus\overline{F}$. It is plain that $F$ and $\overline{F}$ are $G$-invariant. Let
$\mathcal{C}_{0}=(v_{1},\dots, v_{m})$ be a basis of $F$. So $\mathcal{C}: =
(v_{1},\dots, v_{m};\ \overline{v_{1}},\dots, \overline{v_{m}})$ is a basis of $\mathbb{C}^{2m}$. Let $R$ be the matrix of basis
change from  $\mathcal{B}_{0}$  to
$\mathcal{C}$. Then for every $B\in G$, $R^{-1}BR =  \mathrm{diag}(B^{\prime}_{1}, \overline{B^{\prime}_{1}})
\in \textrm{GL}(2m, \mathbb{C}),$
where $B^{\prime}_{1}\in \textrm{GL}(m, \mathbb{C})$.
 We distinguish two cases:
\medskip

- {\it Case 1:} for every $B\in G$, the restriction of $B$ to $F$ has only one eigenvalue.

Write $G^{\prime}_{1}=\{B^{\prime}_{1}: \ B\in G\}$. Then all element of $G^{\prime}_{1}$ has only one eigenvalue. By Lemma ~\ref{L:33},
 there exists a basis
$\mathcal{C}_{0}^{\prime}: =(v_{1}^{\prime},\dots,v_{m}^{\prime})$ of $F$ such that $R^{-1}_{1}G^{\prime}_{1}R_{1}\subset\mathbb{T}^{*}_{m}(\mathbb{C})$,  where $R_{1}$ is the matrix of basis change from
 $\mathcal{C}_{0}$  to  $\mathcal{C}_{0}^{\prime}$. Hence, if  $R^{\prime}= \textrm{diag}(R_{1}, \overline{R_{1}})$ then
$G^{\prime}=
(R^{\prime})^{-1}R^{-1}GRR^{\prime}\subset \mathbb{T}^{*}_{m}(\mathbb{C})\oplus \mathbb{T}^{*}_{m}(\mathbb{C})$. It follows by
Lemma ~\ref{L:1}, that $QG^{\prime}Q^{-1}\subset\mathbb{B}_{m}^{*}(\mathbb{R})$ and hence
$P_{0}^{-1}GP_{0}\subset\mathbb{B}_{m}^{*}(\mathbb{R})$ where $P_{0}= RR^{\prime}Q^{-1}\in \textrm{GL}(2m, \mathbb{C})$ is
 the matrix of basis change from
   $\mathcal{B}_{0}$ to $\mathcal{C}_{0}^{\prime}= (v_{1}^{\prime},\dots,v_{m}^{\prime}; \overline{v_{1}^{\prime}},\dots,
\overline{v_{m}^{\prime}})$. By Lemma ~\ref{r:01}, if $P$ is the matrix of basis change from  $\mathcal{B}_{0}$  to
\ $\mathcal{C}^{\prime}:= \left(\textrm{Re}(v_{1}^{\prime}), \textrm{Im}(v_{1}^{\prime}),\dots,\textrm{Re}(v_{m}^{\prime}),
\textrm{Im}(v_{m}^{\prime})\right)$ then $P\in \textrm{GL}(2m, \mathbb{R})$ and
$P^{-1}GP\subset \mathbb{B}^{*}_{m}(\mathbb{R})$.
\medskip

- {\it Case 2:} for every $B\in G$, the restriction of $B$ to $F$  has two distinct nonreal conjugates complex eigenvalues. Then $F$ can be decomposed further so that

$F= F_{1}\oplus\dots\oplus F_{s}$ is the
maximal decomposition associated to all $B\in G$ where the restriction of every element of $G$ to $F_{l}$ has only one eigenvalue. Write
 $\widetilde{F_{l}}=F_{l}\oplus\overline{F_{l}}$. Then we have
$\mathbb{C}^{2m}=
\widetilde{F_{1}}\oplus\dots \oplus\widetilde{F_{s}}$. Let
$\mathcal{C}_{l}: =(v_{l,1},\dots,v_{l,m_{l}};\overline{v_{l,1}},\dots,\overline{v_{l,m_{l}}})$ be a basis of $\widetilde{F_{l}}$. Then
$$\widetilde{\mathcal{C}_{l}}= (\textrm{Re}(v_{l,1}),\textrm{Im}(v_{l,1}),\dots,\textrm{Re}(v_{l,m_{l}}),\textrm{Im}(v_{l,m_{l}}))$$
is a real
basis of $\widetilde{F_{l}}$. Let $P_{l}\in \textrm{GL}(2m_{l}, \mathbb{R})$ denote the matrix of basis change from
$\mathcal{B}_{0,l}$  to $\widetilde{\mathcal{C}_{l}}$, where $\mathcal{B}_{0,l}$ is the canonical basis of $\mathbb{C}^{2m_{l}}$ and $m_{l}=\textrm{dim}(F_{l})$. Therefore by the case 1, for every $1\leq l\leq s$,
 $P_{l}^{-1}G_{/\widetilde{F_{l}}}P_{l}\subset \mathbb{B}^{*}_{m_{l}}(\mathbb{R})$.  Set
$\widetilde{\mathcal{C}}= (\widetilde{\mathcal{C}_{1}},\dots, \widetilde{\mathcal{C}_{s}})$ and let $\widetilde{Q}\in \textrm{GL}(2m, \mathbb{R})$ denote
 the matrix of basis change from
 $\mathcal{B}_{0}$ to $\mathcal{\widetilde{C}}$. It follows that
$$P=\widetilde{Q}\mathrm{diag}(P_{1},\dots, P_{s})\in \textrm{GL}(2m, \mathbb{R})$$

and

$$P^{-1}GP\subset \mathbb{B}^{*}_{m_{1}}(\mathbb{R})\oplus\dots\oplus\mathbb{B}^{*}_{m_{s}}(\mathbb{R}).$$ This completes the proof.
\end{proof}
\
\\
Combining Lemmas ~\ref{L:01}, ~\ref{L:33} and ~\ref{L:333+}, we obtain
\
\\
\begin{prop}\label{p:1} Let $G$ be an abelian subgroup of
GL($n,\mathbb{R})$. Then:
\begin{itemize}
  \item [(i)] $\mathbb{R}^{n} = \underset{k=1}{\overset{r}{\bigoplus}}E_{k}\oplus \underset{l=1}{\overset{s}{\bigoplus}}F_{l}$
for some $r, \ s$ $(0\leq r \leq n, \  0\leq s \leq \frac{n}{2})$, where $E_{k}$ $($resp. $F_{l}$) is a
$G$-invariant vector subspace of $\mathbb{R}^{n}$ of dimension
$n_{k}$ $($resp. $2m_{l})$, $1\leq k \leq r$ $($resp. $1\leq l\leq s)$.
  \item [(ii)] there exists a basis
$\mathcal{C} = (\mathcal{C}_{1},\dots,\mathcal{C}_{r};\mathcal{B}_{1},\dots,\mathcal{B}_{s})$ of
$\mathbb{R}^{n}$ where $\mathcal{C}_{k}$ $($resp. $\mathcal{B}_{l})$ is a basis of $E_{k}$ $($resp. $F_{l}$) such that if $P$
is the matrix of basis change from $\mathcal{B}_{0}$ to $\mathcal{C}$, we
have $P^{-1}GP$ is an abelian subgroup of
$\mathcal{K}_{\eta,r,s}^{*}(\mathbb{R})$, where $\eta =
 (n_{1},\dots,n_{r};\  m_{1},\dots,m_{s})$.
\end{itemize}
\end{prop}
\medskip

\section{\bf Matrix exponential map}
The following results follow from basic properties of
the matrix exponential map. (The proofs of Lemmas ~\ref{L:2} and ~\ref{L:3}, Propositions ~\ref{p:2} and ~\ref{p:3}
are given in Section 9).
\medskip

\begin{lem}\label{L:2}
Let \ $B\in M_{n}(\mathbb{R})$  having one eigenvalue  $\mu$.
Then:
\begin{enumerate}
    \item [(i)] $\textrm{Ker}(B-\mu I_{n }) = \textrm{Ker}(e^{B}-e^{\mu}I_{n}).$
    \item [(ii)] if   $e^{B}\in \mathbb{T}_{n}^{+}(\mathbb{R})$  then
$B\in \mathbb{T}_{n}(\mathbb{R})$.
\end{enumerate}
 \end{lem}
\smallskip

\begin{prop}\label{p:2}  $\textrm{exp}(\mathcal{K}_{\eta,r,s}(\mathbb{R})) = \mathcal{K}^{+}_{\eta,r,s}(\mathbb{R})$.
\end{prop}
\smallskip

\begin{prop}\label{p:3} Let   $A, \ B\in \mathcal{K}_{\eta,r,s}(\mathbb{R})$. If  $e^{A}e^{B} =
e^{B}e^{A}$ then   $AB = BA$.
\end{prop}
\smallskip

\begin{lem} \label{L:3} Let  $M\in \mathbb{T}_{n}(\mathbb{C})$  be nilpotent such
that  $e^{M} = I_{n}$. Then  $M = 0$.
\end{lem}
\smallskip

\begin{lem}[\cite{aAhM06}, Proposition 2.4]\label{L:4} Let $A, \ B\in \mathbb{T}_{n}(\mathbb{C})$ such that
$AB=BA$. If  $e^{A} = e^{B}$  then  $A = B +2ik\pi I_{n}$  for
some \ $k\in \mathbb{Z}$.
\end{lem}
\medskip

\begin{prop}\label{p:4}\
\begin{itemize}
  \item [(i)] Let $A, \ B\in \mathbb{T}_{n}(\mathbb{R})$ such that
$AB=BA$. If  $e^{A} = e^{B}$ then  $A = B$.
  \item [(ii)] Let  $A$,  $B\in \mathbb{B}_{m}(\mathbb{R})$  such that
$AB = BA$. If  $e^{A} = e^{B}$  then $A = B + 2k\pi J_{m}$, for some $k\in
\mathbb{Z}$  where
 $$J_{m} :=\mathrm{diag}(J_{2},\dots,J_{2})\in \textrm{GL}(2m, \ \mathbb{R})
 \ \ \mathrm{and} \ \ J_{2}=
  \left[\begin{array}{cc}
    0 & -1 \\
    1 & 0 \\
  \end{array}\right].$$
\end{itemize}
\end{prop}
\medskip

\begin{proof}(i) Let $\mu_{A}$ (resp.
$\mu_{B}$)  be the only eigenvalue of $A$ (resp. $B$). Write
$\widetilde{A}=A-\mu_{A}I_{n}$ and $\widetilde{B}=B-\mu_{B}I_{n}$. Then $\widetilde{A}$ and $\widetilde{B}$
are nilpotent matrices. Since  $e^{A} = e^{B}$, we have  $e^{\mu_{A}} =
e^{\mu_{B}}$ and then $\mu_{A}= \mu_{B}$. It follows that
$e^{\widetilde{A}} = e^{\widetilde{B}}$. Since  $AB = BA$, we have
$\widetilde{A}\widetilde{B} = \widetilde{B}\widetilde{A}$  and
therefore  $e^{\widetilde{A} -
\widetilde{B}}=e^{\widetilde{A}}e^{-\widetilde{B}} = I_{n}$. By
Lemma ~\ref{L:3},  $\widetilde{A} = \widetilde{B}$  and therefore
$A = B$.
\
\\
(ii) Let  $Q\in \textrm{GL}(2m, \mathbb{C})$  be the matrix of basis change from
$\mathcal{B}_{0}$  to  $\mathcal{C}_{0}$. By Lemma ~\ref{L:1}, $A^{\prime}: = Q^{-1}AQ=
\mathrm{diag}(A^{\prime}_{1}, \ \overline{A^{\prime}_{1}})$  and
$B^{\prime}: = Q^{-1}BQ =
\mathrm{diag}(B^{\prime}_{1}, \ \overline{B^{\prime}_{1}})$  where
$A^{\prime}_{1}, \ B^{\prime}_{1}\in \mathbb{T}_{m}(\mathbb{C})$. Since $e^{A}=e^{B}$ and $AB=BA$, we have
 $A^{\prime}B^{\prime}=B^{\prime}A^{\prime}$ and
\  $e^{A^{\prime}}=e^{B^{\prime}}$, so $e^{A^{\prime}_{1}} =
e^{B^{\prime}_{1}}$  and  $A^{\prime}_{1}B^{\prime}_{1} =
B^{\prime}_{1}A^{\prime}_{1}$. By Lemma ~\ref{L:4},  $A^{\prime}_{1} =
B^{\prime}_{1} +2ik\pi I_{m}$  for some  $k\in \mathbb{Z}$
and so  $A^{\prime} = B^{\prime} + 2ik\pi L_{m}$  where
$L_{m} := \mathrm{diag}(I_{m}, -I_{m})$. It follows that  $A = B +
2k\pi Q(iL_{m})Q^{-1}$.  Write  \ $J_{m}=Q(iL_{m})Q^{-1}$. We see that $J_{m}=\mathrm{diag}(J_{2},\dots,J_{2})$ where
$J_{2}=\left[\begin{array}{cc}
               0 & -1 \\
               1 & 0
             \end{array}
\right]$ and therefore \ $A = B + 2k\pi J_{m}$.
\end{proof}
\medskip

We also require the following result:
\medskip

\begin{prop}\label{p:5}$($\cite{wR}, Proposition $7^{\prime}$, page 17$)$ Let  $A\in M_{n}(\mathbb{R})$.  Then if no two eigenvalues of
$A$  have a difference of the form  2i$\pi k$,
$k\in\mathbb{Z}\backslash\{0\}$, then  $\textrm{exp} : M_{n}(\mathbb{R})\
\longrightarrow\ \textrm{GL}(n, \ \mathbb{R})$  is a local diffeomorphism
at  $A$.
\end{prop}
\medskip

As a consequence:
\begin{cor}\label{C:5} The restriction
$\textrm{exp}_{/\mathbb{T}_{n}(\mathbb{R})}: \
\mathbb{T}_{n}(\mathbb{R})\ \longrightarrow\
\mathbb{T}_{n}^{*}(\mathbb{R})$  is a local diffeomorphism, in particular
it is an open map.
\end{cor}
\
\begin{lem}\label{L:5} Let $G$ be an abelian subgroup of $\mathcal{K}^{*}_{\eta,r,s}(\mathbb{R})$. Then  $\mathrm{g}=\textrm{exp}^{-1}(G)\cap \mathcal{K}_{\eta,r,s}(\mathbb{R})$ is an additive subgroup of $\mathcal{K}_{\eta,r,s}(\mathbb{R})$.
In particular,  for every  $v\in \mathbb{R}^{n}$, $\mathrm{g}_{v}$  is an additive subgroup of  $\mathbb{R}^{n}$.
\end{lem}
\medskip

\begin{proof} If $A,\ B\in \mathrm{g}$,  then  $e^{A}e^{B} = e^{B}e^{A}$
 and $e^{A}, e^{B}\in\mathcal{K}^{*}_{\eta,r,s}(\mathbb{R})\cap G$. By Proposition ~\ref{p:3},  $AB = BA$. So
$e^{A+B} = e^{A}e^{B}\in G$  and hence $A+B\in
\mathcal{K}_{\eta,r,s}(\mathbb{R})$. It follows that  $A+B\in
\mathrm{g}$. Moreover, if $A\in \mathrm{g}$,  then $e^{A}\in G$
and so $e^{-A}=(e^{A})^{-1}\in G$. Hence  $-A\in \mathrm{g}$. This proves the Lemma.
\end{proof}
\bigskip

 Let  $G$  be an abelian subgroup of
$\mathcal{K}^{*}_{\eta,r,s}(\mathbb{R})$. Denote by
\
\\
$\bullet$ $\mathcal{C}(G):= \{A\in\mathcal{K}_{\eta,r,s}(\mathbb{R}):\ AB=BA,
\ \forall\ B\in G \}$
\
\\
$\bullet$ $\mathcal{C}(\mathrm{g}):= \{A\in\mathcal{K}_{\eta,r,s}(\mathbb{R}):\
AB = BA, \ \forall\ B\in \mathrm{g} \}.$
\bigskip

\begin{lem}\label{L:6}Let $G$ be an abelian subgroup of $\mathcal{K}^{*}_{\eta,r,s}(\mathbb{R})$.
 Under the notation above, we have:
\begin{enumerate}
\item [(i)]  $\textrm{exp}(\mathrm{g})= G^{+}$.
\item [(ii)]  $\textrm{exp}(\mathcal{C}(G))= \mathcal{C}(G)\cap
\mathcal{K}_{\eta,r,s}^{+}(\mathbb{R})$
\item [(iii)]  $\mathcal{C}(\mathrm{g})= \mathcal{C}(G^{+})$.
\item[(iv)]  $\mathrm{g}\subset \mathcal{C}(\mathrm{g})$  and  $\mathrm{g}\subset \mathcal{C}(G)$, moreover, all matrix
of $\mathrm{g}$ commute.
\end{enumerate}
\end{lem}
\medskip

\begin{proof} (i) By
Proposition ~\ref{p:2},  $\textrm{exp}(\mathrm{g})\subset
\mathcal{K}^{+}_{\eta,r,s}(\mathbb{R})$,  hence
$\textrm{exp}(\mathrm{g})\subset G^{+}$.
\
\\
Conversely, let  $A\in G^{+}$. By Proposition  ~\ref{p:2}, there exists
$B\in\mathcal{K}_{\eta,r,s}(\mathbb{R})$  so that  $e^{B} = A$.
Hence $B\in \textrm{exp}^{-1}(G)\cap \mathcal{K}_{\eta,r,s}(\mathbb{R}) =
\mathrm{g}$, and then  $A\in \textrm{exp}(\mathrm{g})$. So
$G^{+}\subset \textrm{exp}(\mathrm{g})$, this proves (i).
\\
(ii) Let  $A = e^{B}$  where  $B\in
\mathcal{C}(G)$ and let $C\in G$. Then $BC = CB$  and
therefore  $Ce^{B} = e^{B}C$, or also  $AC = CA$. It follows that
 $A\in \mathcal{C}(G)$. Since $B\in \mathcal{K}_{\eta,r,s}(\mathbb{R})$, so is $A\in \mathcal{K}_{\eta,r,s}^{+}(\mathbb{R})$
  by Proposition 3.2. Conversely, let  $A\in \mathcal{C}(G)\cap
\mathcal{K}_{n,r,s}^{+}(\mathbb{R})$.  By Proposition ~\ref{p:2}, there
exists  $B\in \mathcal{K}_{n,r,s}(\mathbb{R})$  so that
$e^{B} = A$. Let  $C\in G$. Then  $Ce^{B}=e^{B}C$ and hence  $e^{C}e^{B} = e^{B}e^{C}$.
Since  $B, C\in \mathcal{K}_{n,r,s}(\mathbb{R})$, it follows by
Proposition ~\ref{p:3}, that  $BC=CB$. Therefore  $B\in
\mathcal{C}(G)$ and hence  $A\in \textrm{exp}(\mathcal{C}(G))$.
\
\\
(iii) Let  $B\in \mathcal{C}(G^{+})$  and  $A\in \mathrm{g}$. Then by (i),
 $e^{A}\in G^{+}$  and so $e^{A}B = Be^{A}$. It follows that  $e^{A}e^{B} =
e^{B}e^{A}$. Since $A, \ B\in \mathcal{K}_{n,r,s}(\mathbb{R})$, it follows
 by Proposition ~\ref{p:3}, that  $AB = BA$ and therefore $B\in \mathcal{C}(\mathrm{g})$. Hence
$\mathcal{C}(G^{+})\subset \mathcal{C}(\mathrm{g})$. Conversely, let $B\in \mathcal{C}(\mathrm{g})$ and  $A\in G^{+}$.
By $(i)$  there exists  $C\in \mathrm{g}$  so that  $e^{C}=A$. Hence $BC = CB$ and so $Be^{C}=e^{C}B$. It follows that $B\in \mathcal{C}(G^{+})$ and $\mathcal{C}(\mathrm{g})
 \subset \mathcal{C}(G^{+})$.
\
\\
(iv) By Proposition ~\ref{p:3}, all elements of  $\mathrm{g}$  commute, hence
 $\mathrm{g}\subset \mathcal{C}(\mathrm{g})$. Let $B\in\mathrm{g}$ and $A\in G$, so $e^{B}\in G^{+}\subset G$. As $G$ is abelian,
 $Ae^{B}=e^{B}A$, hence $e^{A}e^{B}=e^{B}e^{A}$. Since $A, \ B\in \mathcal{K}_{n,r,s}(\mathbb{R})$, it follows
 by Proposition ~\ref{p:3}, that  $AB = BA$ and therefore $B\in \mathcal{C}(G)$. We conclude that $\mathrm{g}\subset
\mathcal{C}(G)$.
\end{proof}
\medskip
\section{{\bf Some results for subgroup of $\mathcal{K}^{\ast}_{\eta,r,s}(\mathbb{R})$}}

We let
\
\\
 \textbullet \;$$U:=\begin{cases}

\left(\underset{k=1}{\overset{r}{\prod}}\mathbb{R}^{*}\times\mathbb{R}^{n_{k}-1}\right)\times\left(\underset{l=1}{\overset{s}{\prod}}
(\mathbb{R}^{2}\backslash \{(0,0)\})\times \mathbb{R}^{2m_{l}-2}\right), &\ if \ {r\geq 1} \\

 \underset{l=1}{\overset{s}{\prod}}
(\mathbb{R}^{2}\backslash \{(0,0)\})\times \mathbb{R}^{2m_{l}-2}, &\ if \ {r=0}.
  \end{cases}$$
Then  $U$  is dense in $\mathbb{R}^{n}$, connected if $r=0$, and having $2^{r}$ connected components if $r\geq 1$.
If $G$ is an abelian subgroup of $\mathcal{K}_{\eta,r,s}^{*}(\mathbb{R})$ then a simple calculation from the definition yields
that $U$ is a $G$-invariant.
\bigskip

\begin{prop}\label{p:6} Let $G$ be an abelian subgroup of $\mathcal{K}^{*}_{\eta,r,s}(\mathbb{R})$. Then all orbits of
 $G$ in $U$  are minimal in  $U$.
\end{prop}
 \
 \\
To prove Proposition ~\ref{p:6}, we need the following lemmas:
\\
\begin{lem}[\cite{aAhM05}, Corollary 3.3] \label{L:8} Suppose  $G$  is an abelian subgroup of $\mathbb{T}_{n}^{*}(\mathbb{K})$.
Then for every $v, \ w\in
\mathbb{K}^{*}\times \mathbb{K}^{n-1}$  and any sequence
$(A_{m})_{m\in \mathbb{N}}\subset G$  such that
$\underset{m\to+\infty}{\lim}A_{m}v = w$,  we have  $\underset{m\to+\infty}{\lim}A^{-1}_{m}w = v$.
\end{lem}
\
\\
Notice that the notation  $S_{n}(\mathbb{C})$  in \cite{aAhM05}
corresponds to our notation $\mathbb{T}_{n}^{*}(\mathbb{\mathbb{C}})$.
\
\\
\begin{lem}\label{L:001} Let $Q\in \textrm{GL}(2m, \mathbb{C})$  be the matrix of basis change from $\mathcal{B}_{0}$  to  $\mathcal{C}_{0}$.
Then  for every $u=[x_{1}+iy_{1},\dots, x_{m}+iy_{m}]^{T}\in\mathbb{C}^{m}$, we have
  $Q(u,\overline{u})=[x_{1},y_{1},\dots, x_{m},y_{m}]^{T}$. In particular,
 $Q(H)=\mathbb{R}^{2m},$ where $H:=\{(u,\overline{u}): \ u\in \mathbb{C}^{m}\}$.
 \end{lem}
\medskip

\begin{proof} Let $v=(u,\overline{u})\in H$ with $u=[z_{1},\dots,z_{m}]^{T}\in  \mathbb{C}^{m}$. Write $z_{j}=x_{j}+iy_{j}$, $x_{j}, y_{j} \in\mathbb{R}$, $j=1,\dots, m$.
We have $$v=\underset{k=1}{\overset{m}{\sum}}(x_{k}+iy_{k})e_{k}+\underset{k=1}{\overset{m}{\sum}}(x_{k}-iy_{k})e_{m+k}.$$ Hence
 $$Qv =\underset{k=1}{\overset{m}{\sum}}(x_{k}+iy_{k})Qe_{k}+\underset{k=1}{\overset{m}{\sum}}(x_{k}-iy_{k})Qe_{m+k}.$$
Since   $$Qe_{k}= \left\{\begin{array}{c}
u_{k}, \ \ \   if \ \ 1\leq k \leq m \ \ \ \ \ \ \\

\overline{u_{k-m}}, \ \  \ if \ \ m+1\leq k \leq 2m
\end{array}
\right.,$$
\
\\
it follows that

\begin{align*}
Qv& =\underset{k=1}{\overset{m}{\sum}}(x_{k}+iy_{k})\left(\frac{e_{2k-1}-ie_{2k}}{2}\right)+\underset{k=1}{\overset{m}{\sum}}(x_{k}-iy_{k})\left(\frac{e_{2k-1}+ie_{2k}}{2}\right)\\
\ & =\underset{k=1}{\overset{m}{\sum}}\left(x_{k}e_{2k-1}+y_{k}e_{2k}\right)\\
\ &= [x_{1},y_{1},\dots, x_{m},y_{m}]^{T}.
\end{align*}
In particular, we see that  $Q(H)=\mathbb{R}^{2m}$.
 \end{proof}
\bigskip

\begin{lem}\label{L:9} Suppose $G$ is an abelian subgroup of
$\mathbb{B}^{*}_{m}(\mathbb{R})$. Then for all  $v, \ w\in
\left(\mathbb{R}^{2}\backslash\{(0,0)\}\right)\times \mathbb{R}^{2m-2}$
 and any sequence  $(B_{k})_{k\in \mathbb{N}}\subset G$  such
that  $\underset{k\to+\infty}{lim}B_{k}v = w$,  we
have  $\underset{k\to+\infty}{lim}B^{-1}_{k}w = v$.
\end{lem}
\medskip

\begin{proof} Let  $v = [x_{1},y_{1},\dots,x_{m},y_{m}]^{T}$  with $(x_{1},y_{1})\in
\mathbb{R}^{2}\backslash\{(0,0)\}$,  $w\in
\left(\mathbb{R}^{2}\backslash\{(0,0)\}\right)\times \mathbb{R}^{2m-2}$
and a sequence  $(B_{k})_{k\in \mathbb{N}}\subset G$ be such that
$\underset{k\to+\infty}{lim}B_{k}v=w$. Then by Lemma ~\ref{L:001}, we have $$Q^{-1}v=
[x_{1}+iy_{1},\dots,x_{m}+iy_{m}; \ x_{1}-iy_{1},\dots,x_{m}-iy_{m}]^{T}\in
\left(\mathbb{C}^{*}\times \mathbb{C}^{m-1}\right)^{2}$$ where $Q$ is the matrix of basis change from $\mathcal{B}_{0}$ to
$\mathcal{C}_{0}$. Therefore  $Q^{-1}v, \ Q^{-1}w\in\left(\mathbb{C}^{*}\times
\mathbb{C}^{m-1}\right)^{2}$. From
$\underset{k\to+\infty}{lim}B_{k}v=w$,  we have
$$\underset{k\to+\infty}{lim}Q^{-1}B_{k}Q(Q^{-1}v)=Q^{-1}w.$$
By Lemma ~\ref{L:1}, $Q^{-1}GQ\subset \mathbb{T}^{*}_{m}(\mathbb{C})\oplus  \mathbb{T}^{*}_{m}(\mathbb{C})$, so by
 applying Lemma ~\ref{L:8} to the group  $Q^{-1}GQ$, we obtain  $$\underset{k\to+\infty}{lim}(Q^{-1}B_{k}Q)^{-1}(Q^{-1}w)=Q^{-1}v$$ and therefore
$\underset{k\to+\infty}{lim}B^{-1}_{k}w = v$.
\end{proof}
\medskip

\begin{proof}[Proof of Proposition  ~\ref{p:6}] Let $u\in U$ and $v\in \overline{G(u)}\cap U$. Write
  $$u= [x_{1},\dots, x_{r}; \ u_{1},\dots, u_{s}]^{T}\ \  \mathrm{and} \ \  v = [y_{1},\dots, y_{r}; v_{1},\dots, v_{s}]^{T},$$ \ so \   $x_{k}, y_{k}\in
\mathbb{R}^{*}\times\mathbb{R}^{n_{k}}$ \ and  $u_{l},v_{l}\in
(\mathbb{R}^{2}\backslash\{(0,0)\})\times\mathbb{R}^{2m_{l}-2}$,
 $k=1,\dots,r$,  $l=1,\dots,s$. \ Let  $(A_{p})_{p\in \mathbb{N}}\subset G$ so
that $\underset{p\to +\infty}{lim}A_{p}u = v$. Write $$A_{p} = \mathrm{diag}(A_{1,p},\dots, A_{r,p};
B_{1,p},\dots, B_{s,p})$$ where

$A_{k,p}\in
\mathbb{T}_{n_{k}}^{*}(\mathbb{R}), \
 B_{l,p}\in \mathbb{B}_{m_{l}}^{*}(\mathbb{R}), \ k = 1,\dots,r, \ l= 1,\dots,s$. Therefore
$$\underset{p\to +\infty}{lim}A_{k,p}x_{k} = y_{k} \ \
\mathrm{and} \ \ \underset{p\to
+\infty}{lim}B_{l,p}u_{l} = v_{l}.$$ By Lemmas ~\ref{L:8} and ~\ref{L:9}, we have
  $\underset{p\to+\infty}{lim}A^{-1}_{k,p}y_{k} =
 x_{k}$ \  and \ $\underset{p\to+\infty}{lim}B^{-1}_{l,p}v_{l}=u_{l}$. So  $\underset{p\to+\infty}{lim}A^{-1}_{p}v = u$ and hence  $u\in \overline{G(v)}\cap U$.
  We conclude that  $\overline{G(u)}\cap U =\overline{G(v)}\cap U$. This completes the proof.
\end{proof}
\
\begin{cor}\label{C:6} Let $G$ be an abelian subgroup of $\mathcal{K}_{\eta,r,s}^{*}(\mathbb{R})$. If  $G$ has a dense orbit in  $\mathbb{R}^{n}$  then all
orbits of $G$ in $U$  are dense in  $\mathbb{R}^{n}$.
\end{cor}
\smallskip

\begin{proof} Let  $u\in \mathbb{R}^{n}$ so that $\overline{G(u)} = \mathbb{R}^{n}$. Since  $U$  is a  $G$-invariant
open subset of $\mathbb{R}^{n}$, we see that  $u\in U$. Now, for every  $v\in U$, we have, by Proposition ~\ref{p:6},
  $\overline{G(v)}\cap U=\overline{G(u)}\cap U=U$. It follows that  $\overline{G(v)} =
\mathbb{R}^{n}$  since  $\overline{U} = \mathbb{R}^{n}$.
\end{proof}
\medskip
\
\\
Denote by
\begin{itemize}
 \item [$\bullet$]  $C_{u_{0}}$ the connected component of $U$  containing  $u_{0}$.\\
\item [$\bullet$]  $\Gamma$  the subgroup of
$\mathcal{K}^{*}_{\eta,r,s}(\mathbb{R})$  generated by
$(S_{k})_{1\leq k\leq r}$ where $$S_{k}: = \mathrm{diag}\left(\varepsilon_{1,k}I_{n_{1}},\dots,
\varepsilon_{r,k}I_{n_{r}}; \ I_{2m_{1}} ,\dots, I_{2m_{s}}\right)\in \mathcal{K}^{*}_{\eta,r,s}(\mathbb{R}),$$
$$\varepsilon_{i,k}:=
\begin{cases}
  -1, & \ if \ {i=k} \\
  1, & \ if \ {i\neq k}\end{cases},\ 1\leq i, k \leq r.$$
\end{itemize}
\bigskip

\begin{lem}\label{L:7} Let $G$ be an abelian subgroup of $\mathcal{K}_{\eta,r,s}^{*}(\mathbb{R})$. Under the notation above, we have
\begin{itemize}
 \item [(i)]  $C_{u_{0}}: = \left(\underset{k=1}{\overset{r}{\prod}}\mathbb{R}^{*}_{+}\times\mathbb{R}^{n_{k}-1}\right)\times
\left(\underset{l=1}{\overset{s}{\prod}}(\mathbb{R}^{2}\backslash\{(0,0)\})\times \mathbb{R}^{2m_{l}-2}\right)$.\\

 \item [(ii)] for every  $M\in \mathcal{K}_{\eta,r,s}(\mathbb{R})$,
$S_{k}M = MS_{k}, \ k=1,\dots,r$.\\

 \item [(iii)] $U:= \underset{S\in
\Gamma}{\bigcup}S(C_{u_{0}})$.\\

 \item [(iv)] $G^{+}(u_{0}):=G(u_{0})\cap C_{u_{0}}$.\\

 \item [(v)] $S\mathcal{C}(G)\subset \mathcal{C}(G)$ for every $S\in \Gamma$.\\

 \item [(vi)] if $\textrm{ind}(G)=r$ then $G(u_{0})\cap S(C_{u_{0}})\neq\emptyset$ for every $S\in\Gamma$.
\end{itemize}
\end{lem}
\medskip

\begin{proof}Assertions (i), (ii) and (iii) are easier to prove.
\
\\
\textit{Proof of (iv)}: Let
$A=\textrm{diag}(A_{1},\dots,A_{r}; \widetilde{A}_{1},\dots, \widetilde{A}_{s})\in G^{+}= G\cap \mathcal{K}_{\eta,r,s}^{+}(\mathbb{R})$. Then $$A_{k}=\begin{bmatrix}
  \lambda_{k} &  &   & 0 \\
  a^{(k)}_{2,1} & \lambda_{k} &   &  \\
  \vdots &  \ddots & \ddots &  \\
  a^{(k)}_{n_{k},1} & \dots & a^{(k)}_{n_{k},n_{k}-1} & \lambda_{k}
\end{bmatrix}, \ \lambda_{k}>0, \ k=1,\dots,r.$$ Since
$$A_{k}e_{k,1}=[\lambda_{k},a^{(k)}_{2,1},\dots,a^{(k)}_{n_{k},1}]^{T}\in\mathbb{R}^{*}_{+}\times \mathbb{R}^{n_{k}-1}$$
 and $$\widetilde{A}_{l}f_{l,1}\in(\mathbb{R}^{2}\backslash\{(0,0)\})\times \mathbb{R}^{2m_{l}-2},$$ it follows by  (i)  that
  $Au_{0}=[A_{1}e_{1,1},\dots,A_{r}e_{r,1};\ \widetilde{A}_{1}f_{1,1},\dots,\widetilde{A}_{s}f_{s,1}]^{T}\in C_{u_{0}}$
   and so $G^{+}(u_{0})\subset G(u_{0})\cap C_{u_{0}}$. Conversely, if $A=\textrm{diag}(A_{1},\dots,A_{r}; \widetilde{A}_{1},\dots, \widetilde{A}_{s})\in G$
    and $Au_{0} \in C_{u_{0}}$ then  $A_{k}e_{k,1}=[\lambda_{k},a^{(k)}_{2,1},\dots,a^{(k)}_{n_{k},1}]^{T}\in\mathbb{R}^{*}_{+}\times \mathbb{R}^{n_{k}-1}$,
    hence $\lambda_{k}>0$, and so  $A\in G^{+}$. Therefore $G(u_{0})\cap C_{u_{0}}\subset G^{+}(u_{0})$.
This proves that $G^{+}(u_{0})=G(u_{0})\cap C_{u_{0}}$.
\
 \\
 \textit{Proof of (v)}: Let $B=\textrm{diag}(B_{1},\dots,B_{r}; \widetilde{B}_{1},\dots, \widetilde{B}_{s})\in \mathcal{C}(G)$ and
  $$A=\textrm{diag}(A_{1},\dots,A_{r}; \widetilde{A}_{1},\dots, \widetilde{A}_{s})\in G.$$ So $AB=BA$ and thus $A_{k}B_{k}=B_{k}A_{k}$ and
  $\widetilde{A}_{l}\widetilde{B}_{l}=\widetilde{B}_{l}\widetilde{A}_{l}$, $k=1,\dots,r; \ l=1,\dots,s$.
  Let $$S: =S_{j} = \mathrm{diag}\left(\varepsilon_{1,j}I_{n_{1}},\dots,
\varepsilon_{r,j}I_{n_{r}}; \ I_{2m_{1}} ,\dots, I_{2m_{s}}\right), \ \mathrm{for \ some}\ \ j=1,\dots,r.$$
Then $SB= \mathrm{diag}\left(\varepsilon_{1,j}B_{1},\dots,
\varepsilon_{r,j}B_{r}; \widetilde{B}_{1},\dots, \widetilde{B}_{s}\right)$. Since $\varepsilon_{k,j}B_{k}A_{k}=A_{k}\varepsilon_{k,j}B_{k}$, $k=1,\dots,r$, $j=1,\dots,r$, it follows that
 $(SB)A=A(SB)$.
This proves that  $SB\in \mathcal{C}(G)$.
\
\\
\textit{ Proof of (vi)}: There are three cases:
\
 \\
 - If $r=0$, then $\Gamma=\{I_{n}\}$ and obviously we have (vi): G($u_{0})\cap C_{u_{0}}\neq\emptyset$.
\
 \\
 - If $r=1$, then $\Gamma=\{S_{1},I_{n}\}$.\ Here
 $C_{u_{0}}=\mathbb{R}^{*}_{+}\times\mathbb{R}^{n_{1}-1}\times\left(\underset{l=1}{\overset{s}{\prod}}(\mathbb{R}^{2}\backslash\{(0,0)\})\times\mathbb{R}^{2m_{l}-2}\right)$.
  By definition of ind($G)$, there exists $B\in G$ such that $B_{1}$
  has an eigenvalue $\mu_{1}<0$. So $Bu_{0}\in B(C_{u_{0}})\subset
\mathbb{R}^{*}_{-}\times\mathbb{R}^{n_{1}-1}\times\left(\underset{l=1}{\overset{s}{\prod}}(\mathbb{R}^{2}\backslash\{(0,0)\})
\times\mathbb{R}^{2m_{l}-2}\right)=S_{1}(C_{u_{0}})$, and thus $G(u_{0})\cap S_{1}(C_{u_{0}})\neq\emptyset$, this proves (vi).
\
\\
- If $r\notin\{ 0,1\}$, here $C_{u_{0}}=\left(\underset{k=1}{\overset{r}{\prod}}\mathbb{R}^{*}_{+}\times\mathbb{R}^{n_{k}-1}\right)\times
\left(\underset{l=1}{\overset{s}{\prod}}(\mathbb{R}^{2}\backslash\{(0,0)\})\times \mathbb{R}^{2m_{l}-2}\right)$. By definition of ind($G)$,
for every $1\leq k\leq r$ there exists $B^{(k)}\in G$ such that $B^{(k)}_{k}$ has only one eigenvalue
$\mu_{k}<0$ and all its other real eigenvalues $\mu_{i}>0$, $i\neq k$. So
$B^{(k)}(C_{u_{0}})=
\left(\underset{i=1}{\overset{k-1}{\prod}}\mathbb{R}^{*}_{+}\times\mathbb{R}^{n_{i}-1}\right)\times
\left(\mathbb{R}^{*}_{-}\times\mathbb{R}^{n_{k}-1}\right)\times\left(\underset{i=k+1}{\overset{r}{\prod}}\mathbb{R}^{*}_{+}
\times\mathbb{R}^{n_{i}-1}\right)\times
\left(\underset{l=1}{\overset{s}{\prod}}(\mathbb{R}^{2}\backslash\{(0,0)\})\times \mathbb{R}^{2m_{l}-2}\right) =S_{k}(C_{u_{0}})$, thus
$G(u_{0})\cap S_{k}(C_{u_{0}})\neq\emptyset$, for every $1\leq k \leq r$ and so
G($u_{0})\cap S(C_{u_{0}})\neq\emptyset$, for every $S\in\Gamma$. This completes the proof.
\end{proof}
\bigskip

\begin{lem}\label{L:10} $($\cite{aAhM05}, Corollary 1.3$)$. Let $G$ be an abelian subgroup of $GL(n, \mathbb{R})$.
If G has a locally dense orbit O in $\mathbb{R}^{n}$ and $C$ is a
connected component of $U$ meeting $O$ then $O$ is dense in $C$.
\end{lem}
\medskip

\begin{prop}\label{p:7} Let  $G$  be an abelian subgroup of
$\mathcal{K}^{*}_{\eta,r,s}(\mathbb{R})$. Then:
\begin{enumerate}
    \item [(i)] $\overset{\circ}{\overline{G(u_{0})}}\neq \emptyset$  if and only if
      \; $\overset{\circ}{\overline{G^{+}(u_{0})}}\neq
 \emptyset.$
    \item [(ii)] $\overline{G(u_{0})} = \mathbb{R}^{n}$ if and only if \; $\overline{G^{+}(u_{0})}\cap C_{u_{0}}  =
C_{u_{0}}$  and  $\textrm{ind}(G) = r$.
\end{enumerate}
\end{prop}
\medskip

\begin{proof}
(i) Suppose that
$\overset{\circ}{\overline{G^{+}(u_{0})}}\neq \emptyset$. Since
$G^{+}(u_{0})\subset G(u_{0})$,  we see that
$\overset{\circ}{\overline{G(u_{0})}}\neq \emptyset$. Conversely,
suppose that  $\overset{\circ}{\overline{G(u_{0})}}\neq
\emptyset$. Then by Lemma ~\ref{L:10}, $\overline{G(u_{0})}\cap C_{u_{0}}= C_{u_{0}}$. From Lemma ~\ref{L:7}, (i), we have
$\overset{\circ}{\overline{C_{u_{0}}}}\supset C_{u_{0}}$. As  $G^{+}(u_{0})=G(u_{0})\cap C_{u_{0}}$ (Lemma ~\ref{L:7}, (iv)), then
$\overset{\circ}{\overline{G^{+}(u_{0})}}= \overset{\circ}{\overline{C_{u_{0}}}}\supset C_{u_{0}}$, hence $\overset{\circ}{\overline{G^{+}(u_{0})}}\neq\emptyset$.
\
\\
(ii) Suppose that  $\overline{G(u_{0})} = \mathbb{R}^{n}$. By Lemma ~\ref{L:7}, (iv),
$G^{+}(u_{0})=G(u_{0})\cap C_{u_{0}}$, and then
$\overline{G^{+}(u_{0})}\cap C_{u_{0}}=\overline{G(u_{0})\cap C_{u_{0}}}\cap C_{u_{0}}=C_{u_{0}}$.
Now, suppose that  $\textrm{ind}(G) < r$. Then there exists $1 \leq k_{0} \leq r$ such that for every 
$B = \mathrm{diag}(B_{1},\dots,B_{r};\widetilde{B_{1}},\dots,\widetilde{B_{s}})\in G$ with $B_{k}\in 
\mathbb{T}_{n_{k}} (\mathbb{R})$, $k = 1,\dots,r$ having an eigenvalue $\mu_{k}$ and $\widetilde{B_{l}}\in
 \mathbb{B}_{m_{l}}(\mathbb{R})$, $l = 1,\dots,s$, we have $\mu_{k_{0}} > 0$ or 
 $\mu_{i} < 0$ for some $i\neq k_{0}$. Therefore $G(u_{0})\subset \mathbb{R}^{n}\backslash \mathcal{C}^{\prime}_{u_{0},k_{0}}$ 
 where

$\mathcal{C}^{\prime}_{u_{0},k_{0}}:=\left(\underset{i=1}{\overset{k_{0}-1}{\prod}}\mathbb{R}^{*}_{+}\times\mathbb{R}^{n_{i}-1}\right)\times\left(\mathbb{R}^{*}_{-}\times\mathbb{R}^{n_{k_{0}}-1}\right)
\times\left(\underset{i=k_{0}+1}{\overset{r}{\prod}}\mathbb{R}^{*}_{+}\times\mathbb{R}^{n_{i}-1}\right)\times
\left(\underset{l=1}{\overset{s}{\prod}}(\mathbb{R}^{2}\backslash\{(0,0)\})\times \mathbb{R}^{2m_{l}-2}\right)$ and thus $\overline{\mathbb{R}^{n}\backslash \mathcal{C}^{\prime}_{u_{0},k_{0}}}= \mathbb{R}^{n}$, a
contradiction.
\
\\
Conversely, suppose that  $\overline{G^{+}(u_{0})}\cap C_{u_{0}} =
C_{u_{0}}$  and  $\textrm{ind}(G) = r$. We have
$C_{u_{0}} \subset \overline{G^{+}(u_{0})} \subset
\overline{G(u_{0})}$. By Lemma ~\ref{L:7},(vi),   $G(u_{0})\cap
S(C_{u_{0}})\neq\emptyset,$  for every
$S\in\Gamma$. So let  $v\in G(u_{0})\cap \left(S(C_{u_{0}})\right)$ and
 $w = S^{-1}(v)\in C_{u_{0}}$. By Proposition ~\ref{p:6},  $\overline{G(w)}\cap C_{u_{0}} =
\overline{G(u_{0})}\cap C_{u_{0}}$, so
$\overline{G(w)}\cap C_{u_{0}} = C_{u_{0}}.$ By Lemma ~\ref{p:6}, (ii),   $G(u_{0}) =G(v)= G(Sw) = S(G(w))$.
It follows that  $$\overline{G(u_{0})}\cap
S(C_{u_{0}}) = S\left(\overline{G(w)}\cap
C_{u_{0}}\right) = S(C_{u_{0}}),$$ and hence $S(C_{u_{0}})\subset \overline{G(u_{0})}$. As $U =
\underset{S\in\Gamma}{\bigcup}S(C_{u_{0}})$ (Lemma ~\ref{L:7}, (iii)), then  $U\subset \overline{G(u_{0})}$
 and therefore  $\overline{G(u_{0})} = \mathbb{R}^{n}$ since $\overline{U} =
\mathbb{R}^{n}$.
\end{proof}
\bigskip

\begin{cor}\label{C:44} Let  $G$  be an abelian subgroup of $\mathcal{K}_{\eta,r,s}^{*}(\mathbb{R})$. The following are equivalent:
\begin{enumerate}
    \item [(i)]  $\overline{G(u_{0})}= \mathbb{R}^{n}$
      \item [(ii)] $\overset{\circ}{\overline{G(u_{0})}}\neq
 \emptyset$ \ and \ $\mathrm{ind}(G) = r$.
 \end{enumerate}
\end{cor}
\medskip

\begin{proof}
 $(i)\Longrightarrow(ii):$ results from Proposition ~\ref{p:7}, (ii). $(ii)\Longrightarrow (i)$: By
Lemma ~\ref{L:7},(iv),  $\overline{G^{+}(u_{0})}\cap
C_{u_{0}}= \overline{G(u_{0})}\cap
C_{u_{0}}$  and by Lemma ~\ref{L:10},  $\overline{G(u_{0})}\cap
C_{u_{0}}=C_{u_{0}}$. It follows that $\overline{G^{+}(u_{0})}\cap
C_{u_{0}}=C_{u_{0}}$. Since $\mathrm{ind}(G)=r$, it follows by Proposition ~\ref{p:7},(ii) that $\overline{G(u_{0})}=\mathbb{R}^{n}$.
\end{proof}
\medskip

\begin{lem}\label{L:101} Let  $G$  be an abelian subgroup of $\mathcal{K}^{*}_{\eta,r,s}(\mathbb{R})$. Then $G$ has
 a dense $($resp. locally dense$)$ orbit if and only if $G(u_{0})$ is dense $($resp. locally dense$)$.
\end{lem}
\medskip

\begin{proof} Suppose that $G(v)$ is locally dense in $\mathbb{R}^{n}$, for some $v\in \mathbb{R}^{n}$.
 Then $G(v)\cap U\neq\emptyset$, so $v\in U$ since $U$ is a $G$-invariant, dense open set in $\mathbb{R}^{n}$.
By Lemma ~\ref{L:7},(iii), we have $U=\underset{S\in\Gamma}{\bigcup}S(C_{u_{0}})$, hence there is $S\in\Gamma$ such that $v\in S(C_{u_{0}})$.
 Set $v^{\prime}:= S^{-1}v\in C_{u_{0}}$. So, by Lemma ~\ref{L:7},(ii), we have G($v^{\prime})=S^{-1}(G(v))$, hence G($v^{\prime})$ is
a locally dense orbit in $\mathbb{R}^{n}$ meeting $C_{u_{0}}$, it follows that $G(v^{\prime})$ is dense in  $C_{u_{0}}$, by Lemma
~\ref{L:10}, and therefore $u_{0}\in \overline{G(v^{\prime})}$. By Proposition ~\ref{p:6},
   $\overline{G(u_{0})}\cap U=\overline{G(v^{\prime})}\cap U$ and so $C_{u_{0}}\subset\overline{G(u_{0})}$. Therefore G($u_{0})$ is
 locally dense  in $\mathbb{R}^{n}$.

If $G(v)$ is dense in $\mathbb{R}^{n}$ then G($u_{0})$ is dense in $\mathbb{R}^{n}$ by Corollary ~\ref{C:6}.
\end{proof}
\medskip

\section{{\bf Parametrization of a subgroup  of
$\mathbb{T}_{n}^{*}(\mathbb{K})$}}

Assume that  $G$  is a subgroup of
$\mathbb{T}_{n}^{*}(\mathbb{K})$  and  $\mathrm{g} =
\textrm{exp}^{-1}(G)\cap \mathbb{T}_{n}(\mathbb{K})$. For  $n > 2$  the
group  $\mathbb{T}_{n}^{*}(\mathbb{K})$  is non abelian, so the
assumption  $G$  is abelian imposes restrictions on how it is
embedded in $\mathbb{T}_{n}^{*}(\mathbb{K})$. While there is no
general classification of the abelian subgroups of
$\mathbb{T}_{n}^{*}(\mathbb{K})$  for  $n$  large (see Chapter
3, \cite{daSriT68}), under the assumption that  $G$  is
``sufficiently large", there is a special canonical form for the
matrices of  $G$  which yields a parametrization of an  $n$
dimensional subspace  $\varphi(\mathbb{K}^{n})\subset
\mathbb{T}^{*}_{n}(\mathbb{K})$  containing  $G$.
\\
Recall that for a matrix  $B\in\mathbb{T}_{n}(\mathbb{K})$, the matrix
 $\widetilde{B} = (B - \mu_{B}I_{n})$  where  $\mu_{B}$
denote the unique eigenvalue of  $B$,  is a singular matrix, so
has range of dimension at most  $n-1$.
\\
Introduce the vector subspaces of  $\mathbb{K}^{n}$  generated
by the ranges of all the singular matrices  $\widetilde{B}$  for
$B\in G$  (resp.  $\mathrm{g}$)

$$F_{G} := \textrm{vect}\left \{\widetilde{B}e_{i} : \
 B\in G,\ \
1\leq i\leq n-1\ \right\}$$

$$F_{\mathrm{g}} := \textrm{vect}\left \{\widetilde{B}e_{i} :  \
 B\in \mathrm{g},\ \
1\leq i \leq n-1\ \right\}$$
\
\\
Denote by rank$(F_{G})$ (resp. rank$(F_{\mathrm{g}})$) the rank of  $F_{G}$ (resp. $F_{\mathrm{g}}$).
\bigskip

\begin{lem}\label{L:12}$($\cite{aAhM05}, Lemma 2.3$)$ Let $G$  be an abelian subgroup of
 $\mathbb{T}_{n}^{*}(\mathbb{K})$. Under the notation above, let $r=\mathrm{rank}(F_{G})$, $1\leq r \leq n-1$ and
  $(v_{1},\dots,v_{r})$ be a basis of $F_{G}$. Then for every
$u\in \mathbb{K}^{n}$, the vector subspace

$H_{u}:=\mathrm{vect}(u,v_{1},\dots,v_{r})$  is $G$-invariant. In particular, $F_{G}:=H_{0}$ is $G$-invariant.
\end{lem}
\bigskip

\begin{prop} \label{p:8} Let  $G$ be an abelian subgroup of $\mathbb{T}_{n}^{*}(\mathbb{K})$. If
$\mathrm{rank}(F_{G})=n-1$ $($resp.
$\mathrm{rank}(F_{\mathrm{g}})=n-1)$  then there exist injective
linear maps

$\varphi: \mathbb{K}^{n } \ \longrightarrow \
\ \mathbb{T}_{n}(\mathbb{K})$,  $($resp. $\psi: \ \mathbb{K}^{n } \
\longrightarrow \ \mathbb{T}_{n}(\mathbb{K}))$  such that
\begin{itemize}
  \item [(i)] for
every \; $v\in\mathbb{K}^{n}$,  $\varphi(v)e_{1} = v$ $($resp.
$\psi(v)e_{1} = v)$.
  \item [(ii)] $\mathcal{C}(G)\subset \varphi(\mathbb{K}^{n})$ $($resp. $\mathcal{C}(\mathrm{g})\subset\psi(\mathbb{R}^{n}))$.
  \end{itemize}
\end{prop}
\bigskip

\begin{proof} For $\mathbb{K}=\mathbb{C}$, the proposition is proved in (\cite{aAhM06}, Proposition 5.1). For $\mathbb{K}=\mathbb{R}$,
the proof is similar by the same methods.
\end{proof}
\medskip

Condition (i) asserts that the projection of the embedding
$\varphi$ (resp.  $\psi$)  to the first column of the matrix is
the identity map.
\bigskip

\begin{cor}\label{C:7} \ Under the hypothesis of Proposition ~\ref{p:8}, we have:
\begin{enumerate}
    \item [(i)] $\varphi(Be_{1}) = B$ $($resp.  $\psi(Be_{1}) = B)$  for every  $B\in \mathcal{C}(G)$ $($resp. $\mathcal{C}(\mathrm{g}))$.
    \item [(ii)] $\varphi(G(e_{1})) = G$ $($resp. \ $\psi(G(e_{1})) = G)$.
    \item [(iii)] $\varphi(\mathrm{g}_{e_{1}}) = \mathrm{g}$ $($resp.  $\psi(\mathrm{g}_{e_{1}}) = \mathrm{g})$.
\end{enumerate}
\end{cor}
\medskip

\begin{proof}
(i) Suppose that  $B\in \mathcal{C}(G)$. By Proposition ~\ref{p:8},  there exists
 $v\in \mathbb{K}^{n}$ such that  $\varphi(v)=B$  and so $\varphi(v)e_{1} = v$. Then  $Be_{1} = v$  and  $\varphi(Be_{1})
= B$. Analogous by Proposition ~\ref{p:8}, if $B\in \mathcal{C}(\mathrm{g})$, we have $\psi(Be_{1})=B$.
\
\\
Assertions (ii) follows from (\cite{aAhM05}, Lemma 4.2, iii)) if $\mathbb{K}=\mathbb{C}$ since $\mathcal{C}(\mathrm{g})=\mathcal{C}(G)$.
If $\mathbb{K}=\mathbb{R}$, it follows from Lemma ~\ref{L:6}, (iii) since $G\subset\mathcal{C}(G)$ (resp. $G\subset\mathcal{C}(\mathrm{g})$).

Assertion (iii): If  $\mathbb{K}=\mathbb{R}$ then by (Lemma ~\ref{L:6}, (iv)), $\mathrm{g}\subset \mathcal{C}(\mathrm{g})$ and
$\mathrm{g}\subset \mathcal{C}(G)$ and so by (i),
$\psi(\mathrm{g}_{e_{1}})=\mathrm{g}$ and $\varphi(\mathrm{g}_{e_{1}})=\mathrm{g}$. For $\mathbb{K}=\mathbb{C}$, it is obvious.
\end{proof}
\medskip

\section{Locally dense orbit for subgroups of  $\mathbb{T}^{*}_{n}(\mathbb{K})$}
\
Note that if $G$ is an abelian subgroup of $\mathbb{T}^{*}_{n}(\mathbb{K})$, $U=\mathbb{K}^{*}\times \mathbb{K}^{n-1}$.
\begin{lem}\label{L:13} Let  $G$  be an abelian subgroup of
$\mathbb{T}^{*}_{n}(\mathbb{K})$. If  $\overset{\circ}{\overline{G(e_{1})}}\neq\emptyset$ $($resp.
$\overset{\circ}{\overline{\mathrm{g}_{e_{1}}}}\neq\emptyset)$,  then $\mathrm{rank}(F_{G})=n-1$ $($resp.
 $\mathrm{rank}(F_{\mathrm{g}})=n-1)$.
  \end{lem}
\medskip

\begin{proof} Suppose that  $\overset{\circ}{\overline{G(e_{1})}}\neq\emptyset$. Let $H_{e_{1}}$ be
the vector subspace of $\mathbb{K}^{n}$ generated by $e_{1}$ and $F_{G}$. So by Lemma ~\ref{L:12}, $H_{e_{1}}$ is
$G$-invariant. Hence  $\overline{G(e_{1})}\subset H_{e_{1}}$ and therefore $\overset{\circ}{H_{e_{1}}}\neq\emptyset$.
Hence, $H_{e_{1}}=\mathbb{K}^{n}$ and so $\mathrm{rank}(F_{G})=n-1$.
\
\\
The same proof is true for $F_{g}$ in place of $F_{G}$.
\end{proof}
\smallskip

\begin{lem}\label{L:013} Let $G$ be an abelian subgroup of $\mathbb{T}^{*}_{n}(\mathbb{C})$. Then every locally dense orbit of $G$ is
 dense in $\mathbb{C}^{n}$.
\end{lem}
\medskip

\begin{proof}If $O$ is a locally dense orbit in $\mathbb{C}^{n}$ (i.e.
$\overset{\circ}{\overline{O}}\neq\emptyset$) then $O\subset U$.  Then $\overline{O}\cap U$ is a nonempty
closed subset in $U$. Let's show that $\overline{O}\cap U$ is open in $U$. Let $v \in\overline{O}\cap U$. Since $O$ is
minimal in $U$ (\cite{aAhM06}, Corollary 3.3) then $\overline{O}\cap U =\overline{G(v)}\cap U$. So,
$\overset{\circ}{\overline{O}}\cap U =\overset{\circ}{\overline{G(v)}}\cap U\neq\emptyset$.
Then $v \in\overset{\circ}{\overline{G(v)}}\cap U\subset \overline{O}\cap U$. Since $U$ is connected, $\overline{O}\cap U=U$,
so $\mathbb{C}^{n}=\overline{U}\subset\overline{O}$.
\end{proof}
\medskip

\begin{lem}\label{LL:13} Let  $G$  be an abelian subgroup of
$\mathbb{T}^{*}_{n}(\mathbb{K})$. If  $\overset{\circ}{\overline{G(e_{1})}}\neq\emptyset$ $($resp.
$\overset{\circ}{\overline{\mathrm{g}_{e_{1}}}}\neq\emptyset)$ then there exists an isomorphism $\varphi$ $($resp. $\psi)$ from
$\mathbb{K}^{n}$ to $\mathcal{C}(G)$. In particular,  $\mathcal{C}(G)=\varphi(\mathbb{K}^{n})$ $($resp. $\mathcal{C}(G)=\psi(\mathbb{K}^{n}))$.
\end{lem}
\medskip

\begin{proof} Suppose that
$\overset{\circ}{\overline{G(e_{1})}}\neq\emptyset$. Then by Lemma ~\ref{L:13} and
Proposition ~\ref{p:8}, there exists an injective linear map  $\varphi\ :
\ \mathbb{K}^{n}\longrightarrow\ \mathbb{T}_{n}(\mathbb{K})$
such that  $\mathcal{C}(G)\subset \varphi(\mathbb{K}^{n})$. Let's
prove the inclusion  $\varphi(\mathbb{K}^{n})\subset \mathcal{C}(G)$:

Case $\mathbb{K}=\mathbb{C}$: By Corollary ~\ref{C:7}, (ii),  we have  $\varphi(G(e_{1})) = G$. Recall that here, $u_{0}=e_{1}$,
$U=\mathbb{C}^{*}\times\mathbb{C}^{n-1}$ and by Lemma \ref{L:013}, $G(e_{1})$ is dense in $U$, that is $U\subset \overline{G(e_{1})}$. Then since $\varphi$ is continuous,
we have $\varphi(U)\subset
\varphi(\overline{G(e_{1})})\subset\overline{\varphi(G(e_{1}))} =
\overline{G}$. Since $\mathcal{C}(G)$ is a vector subspace of $M_{n}(\mathbb{C})$, $\overline{G}\subset \mathcal{C}(G)$ and thus
 $\varphi(U)\subset \mathcal{C}(G)$. Since  $\overline{U} = \mathbb{C}^{n}$,
$\varphi\left(\mathbb{C}^{n}\right) =
\varphi\left(\overline{U}\right)\subset\overline{\varphi(U)}\subset
\mathcal{C}(G)$. As a consequence,  $\mathcal{C}(G)=\varphi(\mathbb{C}^{n})$.

Case $\mathbb{K}=\mathbb{R}$. By Corollary ~\ref{C:7}, (ii),  we have  $\varphi(G(e_{1})) = G$.  Recall that here, $u_{0}=e_{1}$, $U=\mathbb{R}^{*}\times\mathbb{R}^{n-1}$ and
 $C_{e_{1}}=\mathbb{R}^{*}_{+}\times\mathbb{R}^{n-1}$ is the connected component of $U$ containing $e_{1}$.
 By Lemma ~\ref{L:10}, $G(e_{1})$ is dense in $C_{e_{1}}$ hence $C_{e_{1}}\subset \overline{G(e_{1})}$. Since $\varphi$  is continuous, we have
$\varphi(C_{e_{1}})\subset
\varphi(\overline{G(e_{1})})\subset\overline{\varphi(G(e_{1}))} =
\overline{G}$.  Since $\mathcal{C}(G)$ is a vector subspace of $M_{n}(\mathbb{R})$, $\overline{G}\subset \mathcal{C}(G)$ and it follows that
 $\varphi(C_{e_{1}})\subset \mathcal{C}(G)$. Since $\varphi$ is linear,
 $\varphi(-C_{e_{1}}) = -\varphi(C_{e_{1}})\subset -\mathcal{C}(G)$. As $-\mathcal{C}(G) = \mathcal{C}(G)$ and $U = (-C_{e_{1}})\cup C_{e_{1}}$, it follows that
 $$\varphi(U) = \varphi(-C_{e_{1}}\cup C_{e_{1}}) = \varphi(-C_{e_{1}})\cup \varphi(C_{e_{1}})\subset \mathcal{C}(G).$$
  Since  $\overline{U} = \mathbb{R}^{n}$,
$\varphi\left(\mathbb{R}^{n}\right) =
\varphi\left(\overline{U}\right)\subset\overline{\varphi(U)}\subset
\mathcal{C}(G)$. As a consequence,  $\mathcal{C}(G)=\varphi(\mathbb{R}^{n})$.

The same proof is given for $\psi$.
\end{proof}

\begin{cor}\label{p:9} Let  $G$  be an abelian subgroup of
$\mathbb{T}^{*}_{n}(\mathbb{R})$. Then:
\begin{itemize}
  \item [(1)] If  $\overset{\circ}{\overline{G(e_{1})}}\neq\emptyset$  then  $f:=\varphi^{-1}\circ
\textrm{exp}_{/\mathbb{T}_{n}(\mathbb{R})}\circ\varphi: \mathbb{R}^{n}\
\longrightarrow \mathbb{R}^{n}$  is well defined and satisfies
\begin{itemize}
  \item [(i)] $f$ is a continuous open map
  \item[(ii)] $f(Be_{1})=e^{B}e_{1}$ \ for every $B\in \mathcal{C}(G)$. In particular, $f(\mathrm{g}_{e_{1}})=G^{+}(e_{1})$.
  \item [(iii)] $f^{-1}(G^{+}(e_{1})) = \mathrm{g}_{e_{1}}$.
 \item [(iv)]  $f(\mathbb{R}^{n})= C_{e_{1}} = \mathbb{R}^{*}_{+}\times\mathbb{R}^{n-1}$.
\end{itemize}

  \item [(2)] If
$\overset{\circ}{\overline{\mathrm{g}_{e_{1}}}}\neq\emptyset$  then  $h:=\psi^{-1}\circ
\textrm{exp}_{/\mathbb{T}_{n}(\mathbb{R})}\circ\psi\ : \mathbb{R}^{n}\
\longrightarrow \mathbb{R}^{n}$  is well defined and satisfies
\begin{itemize}
  \item [(i)] $h$ is a continuous open map
  \item[(ii)] $h(Be_{1})=e^{B}e_{1}$  for every $B\in \mathcal{C}(\mathrm{g})$. In particular, $h(\mathrm{g}_{e_{1}})=G^{+}(e_{1})$.
\item [(iii)] $h(\mathbb{R}^{n}) =  C_{e_{1}} = \mathbb{R}^{*}_{+}\times\mathbb{R}^{n-1}$.
\end{itemize}
\end{itemize}
\end{cor}

\begin{proof} {\it Proof of $(1)$}. (i): By Lemma ~\ref{LL:13},  $\varphi:  \ \mathbb{R}^{n } \ \longrightarrow \ \
\mathcal{C}(G)$ is an isomorphism. By Lemma ~\ref{L:6}, (ii),   $\mathrm{exp}(\mathcal{C}(G))\subset \mathcal{C}(G)$, \
so Lemma \ref{LL:13} implies that  $\mathrm{exp}_{/\mathbb{T}_{n}(\mathbb{R})}(\varphi(\mathbb{R}^{n}))\subset
\varphi(\mathbb{R}^{n})$. Thus the map  $f: = \varphi^{-1}\circ
\mathrm{exp}_{/\mathbb{T}_{n}(\mathbb{R})}\circ\varphi: \mathbb{R}^{n}\
\longrightarrow \mathbb{R}^{n}$  is well defined.
\
\\
(ii): By Corollary
~\ref{C:5},  $\textrm{exp}_{/\mathbb{T}_{n}(\mathbb{R})}: \mathbb{T}_{n}(\mathbb{R})\longrightarrow
\mathbb{T}^{*}_{n}(\mathbb{R})$ is a local diffeomorphism. Hence
$f$ is a local diffeomorphism and therefore $f$ is a continuous open map. By Corollary ~\ref{C:7}, $\varphi(Be_{1})=B$,
for every $B\in \mathcal{C}(G)$. Therefore, for every $B\in \mathcal{C}(G)$, we have

\begin{align*}
  f(Be_{1}) & =\varphi^{-1}\left[
\textrm{exp}_{/\mathbb{T}_{n}(\mathbb{R})}\left(B)\right)\right]
\\
  \  &  =\varphi^{-1}\left(e^{B}\right)\\
  \ & =e^{B}e_{1}.
\end{align*}
(iii): Since  $\varphi(G^{+}(e_{1}))=G^{+}$, $\varphi^{-1}(\mathrm{g})=\mathrm{g}_{e_{1}}$ and by Lemma ~\ref{L:6},
$\textrm{exp}^{-1}_{/\mathbb{T}_{n}(\mathbb{R})}(G^{+})=\mathrm{g}$, it follows that
\begin{align*}
  f^{-1}(G^{+}(e_{1})) & =\varphi^{-1}\left[
\mathrm{exp}^{-1}_{/\mathbb{T}_{n}(\mathbb{R})}\left(\varphi(G^{+}(e_{1}))\right)\right]
\\
  \  &  =\varphi^{-1}\left(
\textrm{exp}^{-1}_{/\mathbb{T}_{n}(\mathbb{R})}(G^{+})\right)\\
 \ & =\varphi^{-1}(\mathrm{g})\\
 \ & =\mathrm{g}_{e_{1}}.
\end{align*}
\
\\
(iv): First, we have $\varphi^{-1}(\mathcal{C}(G)\cap \mathbb{T}_{n}^{+}(\mathbb{R}))=\mathbb{R}^{*}_{+}\times\mathbb{R}^{n-1}$: indeed,
by Proposition ~\ref{p:8},(i), for every $v=[v_{1},\dots,v_{n}]^{T}\in\mathbb{R}^{n}$, we have
$\varphi(v)\in\mathbb{T}_{n}(\mathbb{R})$ and $\varphi(v)e_{1}=v$, so $\varphi(v)$ has the following form
$$\varphi(v)=\left[\begin{array}{cccc}
                  v_{1} & \ & \ & 0 \\
                  v_{2} & v_{1} & \ & \ \\
                  \vdots & * & \ddots & \ \\
                  v_{n} & * & * & v_{1}
                \end{array}
\right].$$
It follows that $\varphi(v)\in\mathbb{T}^{+}_{n}(\mathbb{R})$ if and only if $v_{1}>0$, that is
$v\in\mathbb{R}^{*}_{+}\times\mathbb{R}^{n-1}$.
By Lemma ~\ref{LL:13}, $\varphi(\mathbb{R}^{n})=\mathcal{C}(G)$, it follows that
$\varphi^{-1}(\mathcal{C}(G)\cap \mathbb{T}_{n}^{+}(\mathbb{R}))=\mathbb{R}^{*}_{+}\times\mathbb{R}^{n-1}$.\

Now, as $\varphi:\mathbb{R}^{n}\longrightarrow \mathcal{C}(G)$ is an isomorphism then
\begin{align*}
  h(\mathbb{R}^{n}) & =\varphi^{-1}\left[\textrm{exp}_{/\mathbb{T}_{n}(\mathbb{R})}\left(\mathcal{C}(G)\right)\right] \\
 \ & =\varphi^{-1}(\mathcal{C}(G)\cap \mathbb{T}_{n}^{+}(\mathbb{R})), \textrm{by Lemma}~\ref{L:6},(ii)\\
 \ & =\mathbb{R}^{*}_{+}\times\mathbb{R}^{n-1}.
\end{align*}
\
\\
{\it Proof of $(2)$}. (i): Analogous to the proof of $(1)$, the map $$h:=\psi^{-1}\circ
\mathrm{exp}_{/\mathbb{T}_{n}(\mathbb{R})}\circ\psi: \mathbb{R}^{n}\
\longrightarrow \mathbb{R}^{n}$$  is well defined and it is a local diffeomorphism, hence $h$ is an open map.
\
\\
(ii): For every $B\in \mathcal{C}(\mathrm{g})$  we have
 $e^{B}\in \mathcal{C}(G^{+})\cap \mathbb{T}_{n}^{+}(\mathbb{R})$ since $\mathcal{C}(\mathrm{g})= \mathcal{C}(G^{+})$ (Lemma  ~\ref{L:6}, (iii)) and
by applying Lemma ~\ref{L:6}, (ii) to $G^{+}$. As by
Corollary ~\ref{C:7}, (i), $\psi(Be_{1})=B$, we obtain:
\begin{align*}
  h(Be_{1}) & =\psi^{-1}\left[
\textrm{exp}_{/\mathbb{T}_{n}(\mathbb{R})}\left(B)\right)\right]
\\
  \  & =\psi^{-1}(e^{B})\\
  \ & =e^{B}e_{1}.
\end{align*}
Hence  $$h(\mathrm{g}_{e_{1}})= \textrm{exp}(\mathrm{g})e_{1}=G^{+}(e_{1}), \ \ \ \mathrm{by \ Lemma}~\ref{L:6},(i).$$
\
\\
(iii): we apply the same proof as for (1), (iv) above.
\end{proof}
\medskip

\begin{cor}\label{p:97} Let  $G$  be an abelian subgroup of
$\mathbb{T}^{*}_{n}(\mathbb{C})$. Then:
\begin{itemize}
  \item [(1)] If  $\overset{\circ}{\overline{G(e_{1})}}\neq\emptyset$  then  $f:=\varphi^{-1}\circ
\textrm{exp}_{/\mathbb{T}_{n}(\mathbb{C})}\circ\varphi: \mathbb{C}^{n}\
\longrightarrow \mathbb{C}^{n}$  is well defined and satisfies
\begin{itemize}
  \item [(i)] $f$ is a continuous open map
  \item[(ii)] $f(Be_{1})=e^{B}e_{1}$ \ for every $B\in \mathcal{C}(G)$. In particular, $f(\mathrm{g}_{e_{1}})=G(e_{1})$.
  \item [(iii)] $f^{-1}(G(e_{1})) = \mathrm{g}_{e_{1}}$.
 \item [(iv)]  $f(\mathbb{C}^{n})= \mathbb{C}^{*}\times\mathbb{C}^{n-1}$.
\end{itemize}

  \item [(2)] If
$\overset{\circ}{\overline{\mathrm{g}_{e_{1}}}}\neq\emptyset$  then  $h:=\psi^{-1}\circ
\textrm{exp}_{/\mathbb{T}_{n}(\mathbb{C})}\circ\psi\ : \mathbb{C}^{n}\
\longrightarrow \mathbb{C}^{n}$  is well defined and satisfies
\begin{itemize}
  \item [(i)] $h$ is a continuous open map
  \item[(ii)] $h(Be_{1})=e^{B}e_{1}$  for every $B\in \mathcal{C}(\mathrm{g})$. In particular, $h(\mathrm{g}_{e_{1}})=G(e_{1})$.
\item [(iii)] $h(\mathbb{C}^{n}) = \mathbb{C}^{*}\times\mathbb{C}^{n-1}$.
\end{itemize}
\end{itemize}
\end{cor}
\medskip

\begin{proof} {\it Proof of $(1)$.} Since $\overset{\circ}{\overline{G(e_{1})}}\neq\emptyset$ then by Lemma~\ref{L:013}, $\overline{G(e_{1})}=\mathbb{C}^{n}$.
 So (by Corollary 6.5, in ~\cite{aAhM06}), $f:=\varphi^{-1}\circ
\textrm{exp}_{/\mathbb{T}_{n}(\mathbb{C})}\circ\varphi: \mathbb{C}^{n}\
\longrightarrow \mathbb{C}^{n}$  is well defined and satisfies (i) and (iii).
\
\\
(ii): By (\cite{aAhM06}, Corollary 3.7),  $\textrm{exp}_{/\mathbb{T}_{n}(\mathbb{C})}: \mathbb{T}_{n}(\mathbb{C})\longrightarrow
\mathbb{T}^{*}_{n}(\mathbb{C})$ is a local diffeomorphism. Hence
$f$ is a local diffeomorphism and therefore $f$ is a continuous open map. By Corollary ~\ref{C:7}, $\varphi(Be_{1})=B$,
for every $B\in \mathcal{C}(G)$. Therefore, for every $B\in \mathcal{C}(G)$, we have

\begin{align*}
  f(Be_{1}) & =\varphi^{-1}\left[
\textrm{exp}_{/\mathbb{T}_{n}(\mathbb{R})}\left(B)\right)\right]
\\
  \  &  =\varphi^{-1}\left(e^{B}\right)\\
  \ & =e^{B}e_{1}.
\end{align*}
\
\\
(iv): First, we have $\varphi^{-1}(\mathcal{C}(G)\cap \mathbb{T}_{n}^{*}(\mathbb{C}))=\mathbb{C}^{*}\times\mathbb{C}^{n-1}$: indeed,
by Proposition ~\ref{p:8},(i), for every $v=[v_{1},\dots,v_{n}]^{T}\in\mathbb{C}^{n}$, we have
$\varphi(v)\in\mathbb{T}_{n}(\mathbb{C})$ and $\varphi(v)e_{1}=v$, so $\varphi(v)$ has the following form
$$\varphi(v)=\left[\begin{array}{cccc}
                  v_{1} & \ & \ & 0 \\
                  v_{2} & v_{1} & \ & \ \\
                  \vdots & * & \ddots & \ \\
                  v_{n} & * & * & v_{1}
                \end{array}
\right].$$
It follows that $\varphi(v)\in\mathbb{T}^{*}_{n}(\mathbb{C})$ if and only if $v_{1}\neq 0$, that is
$v\in\mathbb{C}^{*}\times\mathbb{C}^{n-1}$.
By Lemma ~\ref{LL:13}, $\varphi(\mathbb{C}^{n})=\mathcal{C}(G)$, it follows that
$\varphi^{-1}(\mathcal{C}(G)\cap \mathbb{T}_{n}^{*}(\mathbb{C}))=\mathbb{C}^{*}\times\mathbb{C}^{n-1}$.\

Now, as $\varphi:\mathbb{C}^{n}\longrightarrow \mathcal{C}(G)$ is an isomorphism then
\begin{align*}
  h(\mathbb{C}^{n}) & =\varphi^{-1}\left[\textrm{exp}_{/\mathbb{T}_{n}(\mathbb{C})}\left(\mathcal{C}(G)\right)\right] \\
 \ & =\varphi^{-1}(\mathcal{C}(G)\cap \mathbb{T}_{n}^{*}(\mathbb{C})), \textrm{by (\cite{aAhM06}, Lemma 4.2, ii))}\\
 \ & =\mathbb{C}^{*}\times\mathbb{C}^{n-1}.
\end{align*}
\
\\
(2): Similar considerations apply for $\overset{\circ}{\overline{\mathrm{g}_{e_{1}}}}\neq\emptyset$ and by using (Corollary 6.4, in ~\cite{aAhM06}).
\end{proof}
\medskip

\section{{\bf Locally dense orbit for subgroups of  $\mathcal{K}^{*}_{\eta,r,s}(\mathbb{R})$}}
\bigskip

\subsection{Case where $G$ is a subgroup of  $\mathbb{B}_{m}^{*}(\mathbb{R})$}
\
\\
Let $G$ be an abelian subgroup of
 $\mathbb{B}_{m}^{*}(\mathbb{R})$. In this case $G^{+}=G$. Recall that $H=\{(u,\overline{u}): \ u\in
\mathbb{C}^{m}\}\subset\mathbb{C}^{2m}$ and  that for every  $B\in G$, $$B^{\prime}:= Q^{-1}BQ =
\mathrm{diag}\left(B^{\prime}_{1},\ \overline{B^{\prime}_{1}}\right)$$  where \;
 $B^{\prime}_{1}\in \mathbb{T}_{m}^{*}(\mathbb{C})$ and $Q\in \textrm{GL}(2m, \mathbb{C})$ is the
matrix of basis change from $\mathcal{B}_{0}$ to $\mathcal{C}_{0}$. Moreover, $Q^{-1}(\mathbb{R}^{2m})=H$ (see Lemma ~\ref{L:001}). Denote by
 $G^{\prime}_{1}:=\{B^{\prime}_{1}\in\mathbb{T}^{*}_{m}(\mathbb{C}): \ B\in G\}$. Then $G^{\prime}_{1}$ is an abelian subgroup of $\mathbb{T}^{*}_{m}(\mathbb{C})$.
\medskip

\begin{prop}\label{p:10} Let  $G$ be an abelian subgroup of
$\mathbb{B}_{m}^{*}(\mathbb{R})$.
\begin{itemize}
  \item [(1)] If \;
$\overset{\circ}{\overline{G(e_{1})}}\neq\emptyset$ then there exists a map  $f: \mathbb{R}^{2m}\
\longrightarrow \mathbb{R}^{2m}$  satisfying
\begin{itemize}
  \item [(i)] $f$ is continuous and open
  \item[(ii)] $f(Be_{1}) = e^{B}e_{1}$ \ for every $B\in \mathcal{C}(G)$.
  \item [(iii)]$f^{-1}(G(e_{1})) = \mathrm{g}_{e_{1}}$.
 \item [(iv)]  $f(\mathbb{R}^{2m}) = (\mathbb{R}^{2}\backslash\{(0,0)\})\times\mathbb{R}^{2m-2}$.
\end{itemize}

  \item [(2)] If $\overset{\circ}{\overline{\mathrm{g}_{e_{1}}}}\neq\emptyset$
then there exists a map \ $h: \mathbb{R}^{2m}\
\longrightarrow \mathbb{R}^{2m}$  satisfying
\begin{itemize}
\item [(i)] $h$ is continuous and open
  \item [(ii)]
$h(Be_{1})=e^{B}e_{1}$ for every $B\in\mathcal{C}(\mathrm{g})$. In particular,
$h(\mathrm{g}_{e_{1}})=G(e_{1})$.
 \item [(iii)] $h(\mathbb{R}^{2m}) = (\mathbb{R}^{2}\backslash\{(0,0)\})\times\mathbb{R}^{2m-2}$.
\end{itemize}
\end{itemize}
\end{prop}
\smallskip

\begin{proof} Suppose that $\overset{\circ}{\overline{G(e_{1})}}\neq \emptyset$. The proof for
$\overset{\circ}{\overline{\mathrm{g}_{e_{1}}}}\neq\emptyset$  is analogous. By Lemma ~\ref{L:001},  we have
$Q^{-1}(e_{1})=(e^{\prime}_{1},\overline{e^{\prime}_{1}})= (e^{\prime}_{1},e^{\prime}_{1})$,
 where $e^{\prime}_{1}=[1,0,\dots,0]^{T}\in\mathbb{C}^{m}$. Write $G^{\prime}: = Q^{-1}GQ\subset \textrm{GL}(2m,\mathbb{C})$.
Then $G^{\prime}(e^{\prime}_{1}, e^{\prime}_{1})=Q^{-1}(G(e_{1}))=\{(u,\overline{u}):\ u\in G^{\prime}_{1}(e^{\prime}_{1})\}\subset H$. Hence
$\overset{\circ}{\overline{G^{\prime}(e^{\prime}_{1},e^{\prime}_{1})}}\neq\emptyset$ and
 $H$ is $G^{\prime}$-invariant.  Define  $p_{1}: H\longrightarrow
\mathbb{C}^{m}$  by $p_{1}(u,\overline{u})=u$. We have
$p_{1}\left(G^{\prime}(e^{\prime}_{1},e^{\prime}_{1})\right)=G^{\prime}_{1}(e^{\prime}_{1})$. Since $p_{1}$ is open and continuous,
it follows that
$\emptyset\neq p_{1}\left(\overset{\circ}{\overline{G^{\prime}(e^{\prime}_{1},e^{\prime}_{1})}}\right)\subset
\overset{\circ}{\overline{G^{\prime}_{1}(e^{\prime}_{1})}}$,
hence $\overset{\circ}{\overline{G^{\prime}_{1}(e^{\prime}_{1})}}\neq \emptyset$. By Corollary \ref{p:97}, there exists a continuous open map
$f_{1}: \mathbb{C}^{m}\ \longrightarrow
\mathbb{C}^{m}$ that satisfies $f_{1}(B^{\prime}_{1}e^{\prime}_{1})= e^{B^{\prime}_{1}}e^{\prime}_{1}$ for every
$B^{\prime}_{1}\in \mathcal{C}
(G^{\prime}_{1})$ and
$f^{-1}_{1}(G^{\prime}_{1}(e^{\prime}_{1})) =
\mathrm{g^{\prime}_{1}}_{e^{\prime}_{1}}$  where
$\mathrm{g}^{\prime}_{1}=\textrm{exp}^{-1}_{/\mathbb{T}_{m}(\mathbb{C})}(G^{\prime}_{1})$. Let  $\widetilde{f} : H\ \longrightarrow H$ defined by  $\widetilde{f}(u,\overline{u})=(f_{1}(u), \
\overline{f_{1}(u)})$  and set  $f:= Q\circ \widetilde{f}\circ Q^{-1}
: \mathbb{R}^{2m}\ \longrightarrow \mathbb{R}^{2m}$. The map  $\widetilde{f}$
is continuous and open, so is $f$. For every $B\in \mathcal{C}(G)$, we have $B^{\prime}_{1}\in \mathcal{C}
(G^{\prime}_{1})$ and then
 $$f(Be_{1})= Q\circ \widetilde{f}(Q^{-1}(Be_{1}))= Q\left(f_{1}(B^{\prime}_{1}e^{\prime}_{1}),
\overline{f_{1}(B^{\prime}_{1}e^{\prime}_{1})}\right)= Q(e^{B^{\prime}_{1}}e^{\prime}_{1}, \overline{e^{B^{\prime}_{1}}e^{\prime}_{1}})=e^{B}e_{1}.$$ Set $F: = (\widetilde{f})^{-1}(G^{\prime}(e^{\prime}_{1},e^{\prime}_{1}))$. We have
\begin{align*}
F & =\{(z,\overline{z}):\ z\in f_{1}^{-1}(G^{\prime}_{1}(e^{\prime}_{1})\}\\
\ & =\{(z,\overline{z}):\ z\in (\mathrm{g}^{\prime}_{1})_{e^{\prime}_{1}}\}\\
\ & = \{(B^{\prime}_{1}e^{\prime}_{1},\overline{B^{\prime}_{1}e^{\prime}_{1}}):\ B^{\prime}_{1}\in \mathrm{g}^{\prime}_{1}\}\\
\ & = \{B^{\prime}(e^{\prime}_{1},e^{\prime}_{1}):\ B^{\prime}=\mathrm{diag}(B^{\prime}_{1},\overline{B^{\prime}_{1}})\in
Q^{-1}\mathrm{g}Q\}.
\end{align*}
 Then

\begin{align*}
f^{-1}(G(e_{1})) & \ = Q\left((\widetilde{f})^{-1}(Q^{-1}(G(e_{1})))\right)\\
\ & = Q\left((\widetilde{f})^{-1}\left(G^{\prime}(e^{\prime}_{1},e^{\prime}_{1})\right)\right) \\
 \ & =Q(F)\\
  & =\{Be_{1}:\ B\in \mathrm{g}\}\\
 \ & =\mathrm{g}_{e_{1}}
\end{align*}
Finally, $f(\mathbb{R}^{2m})=  Q\circ \widetilde{f}\circ Q^{-1}(\mathbb{R}^{2m})= Q\circ \widetilde{f}(H)= Q(\{(u, \overline{u}): u\in
\mathbb{C}^{\ast}
\times \mathbb{C}^{m-1}\})=  (\mathbb{R}^{2}\backslash\{(0,0)\})\times\mathbb{R}^{2m-2}$.
This completes the proof.
\end{proof}
\bigskip

\subsection{Case where $G$ is a subgroup of \ $\mathcal{K}^{*}_{\eta,r,s}(\mathbb{R})$}
\
\\
Denote by
\
\\
$\bullet$\; $G_{k} = G_{/E_{k}}:= \{M_{k}: \ M\in G\}$ $\left(resp. \ \widetilde{G}_{l}= G_{/F_{l}}:= \{\widetilde{M_{l}}:
\ M\in G\}\right)$, $1\leq k\leq r, \ 1\leq l\leq s$, where $M =\mathrm{diag}( M_{1}, \dots, M_{r}; \ \widetilde{M}_{1}, \dots,
\widetilde{M}_{s})$, \ with \ $M_{k}\in
\mathbb{T}_{n_{k}}^{*}(\mathbb{R})$  and
$\widetilde{M}_{l}\in \mathbb{B}_{m_{l}}^{*}(\mathbb{R})$. Then $G_{k}$ (resp. $\widetilde{G}_{l}$) is an abelian
subgroup of $\mathbb{T}_{n_{k}}^{*}(\mathbb{R})$  $\left(\mathrm{resp}. \
\mathbb{B}_{m_{l}}^{*}(\mathbb{R})\right)$.
\
\\
$\bullet$ $\mathrm{g}_{k}= \mathrm{g}_{/E_{k}}: = \textrm{exp}^{-1}(G_{k})\cap
\mathbb{T}_{n_{k}}(\mathbb{R})$  and  $\widetilde{\mathrm{g}}_{l}= \mathrm{g}_{/F_{l}}: =
\textrm{exp}^{-1}(\widetilde{G}_{l})\cap \mathbb{B}_{m_{l}}(\mathbb{R})$.
\
\\
Recall that  $u_{0} =
[e_{1,1},\dots, e_{r,1}; f_{1,1},\dots, f_{1,s}]^{T}\in
\mathbb{R}^{n}$  where $$e_{k,1} = [1,0,\dots,0]^{T}\in
\mathbb{R}^{n_{k}}, \ f_{l,1} = [1,0,\dots,0]^{T}\in
\mathbb{R}^{2m_{l}}, \ 1\leq k\leq r, 1\leq l\leq s.$$
\medskip

\begin{thm}\label{Th.3} Let  $G$  be an abelian subgroup of $\mathcal{K}^{*}_{\eta,r,s}(\mathbb{R})$.
\begin{itemize}
  \item [(1)] If \;
 $\overset{\circ}{\overline{G(u_{0})}}\neq\emptyset$ then there exists a map $f:\ \mathbb{R}^{n}\
\longrightarrow\ \mathbb{R}^{n}$ satisfying
\begin{itemize}
  \item [(i)] $f$ is continuous and open
  \item[(ii)]  $f(Bu_{0})=e^{B}u_{0}$  for every $B\in \mathcal{C}(G)$.
  \item [(iii)]$f^{-1}(G^{+}(u_{0}))=\mathrm{g}_{u_{0}}$.
 \item [(iv)]  $f(\mathbb{R}^{n}) =  C_{u_{0}}$.
\end{itemize}

  \item [(2)] If $\overset{\circ}{\overline{\mathrm{g}_{u_{0}}}}\neq\emptyset$
then there exists a map $h :\ \mathbb{R}^{n}\
\longrightarrow\ \mathbb{R}^{n}$ satisfying
\begin{itemize}
\item [(i)] $h$ is continuous and open
  \item [(ii)]
$h(Bu_{0})=e^{B}u_{0}$ for every $B\in\mathcal{C}(\mathrm{g})$. In particular,
$h(\mathrm{g}_{u_{0}}) = G^{+}(u_{0})$.
 \item [(iii)]  $h(\mathbb{R}^{n}) =  C_{u_{0}}$.
\end{itemize}
\end{itemize}
\end{thm}
\medskip

\begin{proof} Let's prove the theorem for  $\overset{\circ}{\overline{G(u_{0})}}\neq\emptyset$. The proof for
$\overset{\circ}{\overline{\mathrm{g}_{u_{0}}}}\neq\emptyset$  is analogous. We have
$\overset{\circ}{\overline{G_{/E_{k}}(e_{k,1})}}\neq\emptyset$ and
$\overset{\circ}{\overline{G_{/F_{l}}(f_{l,1})}}\neq\emptyset$. By Corollary~\ref{p:9} (resp. Proposition ~\ref{p:10}),
there exists a continuous open map $$f_{k}: = f_{/E_{k}}:\
\mathbb{R}^{n_{k}}\ \longrightarrow\ \mathbb{R}^{n_{k}} (\textrm{resp}. \
\widetilde{f_{l}}:= \widetilde{f}_{/F_{l}}:\ \mathbb{R}^{2m_{k}}\ \longrightarrow\
\mathbb{R}^{2m_{k}})$$ satisfying, for every $k=1,\dots, r, \ l=1,\dots, s$:
$$f_{k}(B_{k}(e_{k,1}))= e^{B_{k}}e_{k,1}, \ f_{k}^{-1}(G^{+}_{k}(e_{k,1}))=(\mathrm{g}_{k})_{e_{k,1}} \textrm{and} \
f_{k}(\mathbb{R}^{n_{k}}) = \mathbb{R}^{*}_{+}\times\mathbb{R}^{n_{k}-1}$$
$$\left(\mathrm{resp.} \
\widetilde{f_{l}}(\widetilde{B}_{l}(f_{l,1}))= e^{B_{l}}f_{l,1}, \
 (\widetilde{f_{l}})^{-1}(\widetilde{G}_{l}(f_{l,1}))=(\widetilde{\mathrm{g}}_{l})_{f_{l,1}} \ \mathrm{and}\ \widetilde{f_{l}}(\mathbb{R}^{2m_{l}})
= (\mathbb{R}^{2}\backslash \{(0,0)\})\times\mathbb{R}^{2m_{l}-1}\right),
$$
where $B_{k}=B_{/E_{k}}$ and $\widetilde{B}_{l}=B_{/F_{l}}$. Let
$f:\ \mathbb{R}^{n}\ \longrightarrow\ \mathbb{R}^{n}$ denote the map defined by
 $$f(v)=\left[f_{1}(v_{1}),\dots, f_{r}(v_{r}); \
\widetilde{f_{1}}(\widetilde{v_{1}}),\dots, \widetilde{f_{s}}(\widetilde{v_{s}})\right]^{T}, \ v = [v_{1},\dots, v_{r};\
\widetilde{v_{1}},\dots, \widetilde{v_{s}}]^{T}\in\mathbb{R}^{n}$$
where $v_{k}\in \mathbb{R}^{n_{k}},  \ u_{l}\in \mathbb{R}^{2m_{l}}, \ k=1,\dots, r, \ l=1,\dots, s.$  Thus $f$ is a continuous
open map satisfying
 $$f(Bu_{0})=e^{B}u_{0}, \ f^{-1}(G^{+}(u_{0}))=\mathrm{g}_{u_{0}} \ \mathrm{and} \ f(\mathbb{R}^{n}) =  C_{u_{0}}.$$
\end{proof}
\bigskip

\section{{\bf Proof of main results}}
\smallskip
\emph{ Proof of Theorem  ~\ref{T:1}.} One can assume by Proposition ~\ref{p:1}
that  $G$  is an abelian subgroup of $\mathcal{K}^{*}_{\eta,r,s}(\mathbb{R})$.\
\
\\
$(ii)\Longrightarrow (i)$:  is clear.
\
\\
$(i)\Longrightarrow (ii)$: this follows directly from Lemma ~\ref{L:101}.
\
\\
$(iii)\Longrightarrow (ii)$:
 suppose that
$\overline{\mathrm{g}_{u_{0}}} = \mathbb{R}^{n}$. Then by Theorem
~\ref{Th.3}, there exists a continuous open map
$h:\mathbb{R}^{n}\longrightarrow \mathbb{R}^{n}$  such that
$h(\mathrm{g}_{u_{0}}) = G^{+}(u_{0})$ and $h(\mathbb{R}^{n})= C_{u_{0}}$. Hence, one has:
$$ C_{u_{0}}=h(\overline{\mathrm{g}_{u_{0}}}) \subset
\overline{G^{+}(u_{0})}\subset \overline{G(u_{0})}.$$ Therefore,
$\overset{\circ}{\overline{G(u_{0})}}\neq\emptyset$.
\
\\
$(ii)\Longrightarrow (iii)$:  suppose that  $\overset{\circ}{\overline{G(u_{0})}}\neq\emptyset$. By Theorem ~\ref{Th.3},
there exists a continuous open map
 $f:\mathbb{R}^{n}\longrightarrow \mathbb{R}^{n}$
such that $f^{-1}(G^{+}(u_{0})) = \mathrm{g}_{u_{0}}$. By
Proposition ~\ref{p:7}, (i), $\overset{\circ}{\overline{G^{+}(u_{0})}}\neq\emptyset$. Then by Lemma~\ref{L:10},
$\overline{G^{+}(u_{0})}\cap C_{u_{0}}=C_{u_{0}}$, so $\overset{\circ}{\overline{G^{+}(u_{0})}}
= \overset{\circ}{\overline{ C_{u_{0}}}}\supset C_{u_{0}}$, and by Theorem \ref{Th.3}, (iv), it follows that
$$ \mathbb{R}^{n}=f^{-1}( C_{u_{0}})\subset f^{-1}\left(\overline{G^{+}(u_{0})}\right)\subset \overline{f^{-1}(G^{+}(u_{0}))}
\subset\overline{\mathrm{g}_{u_{0}}}.$$ Hence, $\overline{\mathrm{g}_{u_{0}}} =
\mathbb{R}^{n}$.\qed
\bigskip

\emph{Proof of Corollary  ~\ref{C:1}.} One can assume by Proposition ~\ref{p:1}, that $G\subset \mathcal{K}^{*}_{\eta,r,s}(\mathbb{R})$,
in this case $v_{0}=u_{0}$.
$(i)\Longleftrightarrow (ii)$ follows directly from Lemma ~~\ref{L:101}.
$(iii)\Longrightarrow (ii)$: Suppose that $\overline{\mathrm{g}_{u_{0}}} = \mathbb{R}^{n}$  and $\textrm{ind}(G) = r$.
Then by Theorem ~\ref{T:1},
$\overset{\circ}{\overline{G(u_{0})}}\neq\emptyset$. By
Corollary \ref{C:44}, $\overline{G(u_{0})}=\mathbb{R}^{n}$.
\
\\
$(ii)\Longrightarrow(iii):$ results from Proposition ~\ref{p:7}, (ii).\qed
\bigskip

\emph{Proof of Corollary ~\ref{C:2}}. This follows directly from Corollary ~\ref{C:1} by taking $u_{0}=e_{1}$.
\bigskip

\begin{prop} \label{p:11}Let  $G$  be an abelian subgroup of $\mathcal{K}^{+}_{\eta,r,s}(\mathbb{R})$ and let
  $B_{1},\dots,B_{p}\in \mathcal{K}_{\eta,r,s}(\mathbb{R})$  such that  $e^{B_{1}},\dots, e^{B_{p}}$ generate $G$. Then we have
 $$\mathrm{g}_{u_{0}} =\underset{k=1}{\overset{p}{\sum}}\mathbb{Z}(B_{k}u_{0})+
\underset{k=1}{\overset{s}{\sum}}2\pi\mathbb{Z}f^{(l)}.$$ \qed
 \end{prop}
\bigskip

\begin{proof}
$\bullet$  First we determine  $\mathrm{g}$. Let  $M\in \mathrm{g}$. Then
$$M = \mathrm{diag}(M_{1}, \dots, M_{r};\ \widetilde{M}_{1}, \dots, \widetilde{M}_{s})\in\mathcal{K}_{\eta,r,s}(\mathbb{R})$$ and
$e^{M}\in G$. So \ $e^{M} =e^{k_{1}B_{1}} \dots\ e^{k_{p}B_{p}}$ for some
 $k_{1},\dots,k_{p}\in \mathbb{Z}$. Since $B_{1},\dots,B_{p}\in \mathrm{g}$, they pairwise commute (Lemma ~\ref{L:6}, (iv)).
 Therefore \ $e^{M}= e^{k_{1}B_{1} + \dots + k_{p}B_{p}}$. Write $B_{j}=\mathrm{diag}(B_{j,1},\dots,B_{j,r};\
 \widetilde{B}_{j,1},\dots,\widetilde{B}_{j,s})$, then \ $e^{M_{i}}= e^{k_{1}B_{1,i} + \dots + k_{p}B_{p,i}}$, $i=1,\dots,r$ and
 $e^{\widetilde{M}_{l}}= e^{k_{1}\widetilde{B}_{1,l} + \dots + k_{p}\widetilde{B}_{p,l}}$, $l=1,\dots,s$. Moreover, as
 $M\in\mathrm{g}$, we have $MB_{j}=B_{j}M$ and hence $M_{i}B_{j,i}=B_{j,i}M_{i}$, $i=1,\dots,r$ and
 $\widetilde{M}_{l}\widetilde{B}_{j,l}=\widetilde{B}_{j,l}\widetilde{M}_{l}$, $l=1,\dots,r$, $j=1,\dots,p$.
  It follows by Proposition ~\ref{p:4}, that $M_{i}=k_{1}B_{1,i} + \dots
+ k_{p}B_{p,i}$ \ and  \ $\widetilde{M}_{l}=k_{1}\widetilde{B}_{1,l} +
\dots +  k_{p}\widetilde{B}_{p,l}+2\pi t_{l}J_{m_{l}}$ for some $t_{l}\in\mathbb{Z}$ where
$J_{m_{l}}= \mathrm{diag}(J_{2},\dots, J_{2})\in GL(2_{m_{l}}, \mathbb{R})$ with 
$J_{2} =\left[\begin{array}{cc}
0 & -1 \\
1 & 0 
\end{array}
\right].$

Therefore

\begin{align*}
M & =\mathrm{diag}\left(
   \underset{j=1}{\overset{p}{\sum}}k_{j}B_{j,1},\dots, \underset{j=1}{\overset{p}{\sum}}k_{j}B_{j,r};
   \underset{j=1}{\overset{p}{\sum}}k_{j}\widetilde{B}_{j,1}+2\pi
t_{1}J_{m_{1}},\dots,\underset{j=1}{\overset{p}{\sum}}k_{j}\widetilde{B}_{j,s}+2\pi
t_{s}J_{m_{s}}\right)\\
\ &=\underset{j=1}{\overset{p}{\sum}}k_{j}B_{j}+ \mathrm{diag}(0,\dots,0;\ 2\pi
t_{1}J_{m_{1}},\dots,2\pi
t_{p}J_{m_{p}}).
\end{align*}
\\
Set $$L_{l} := \mathrm{diag}(0,\dots, 0; \widetilde{L_{l,1}},\dots,\widetilde{L_{l,s}})$$
where
\ $$\widetilde{L_{l,i}} = \left\{\begin{array}{cc}
  0\in \mathbb{B}_{m_{i}}(\mathbb{R}) &  if \ \ \ {i\neq l}  \\
  J_{m_{l}}\ \ \ \ \ \ \ \ \ \ \  & \ if \ \ \ {i=l} \\
\end{array}\right.$$
\\
Then we have $\mathrm{diag}(0,\dots,0;\ 2\pi
t_{1}J_{m_{1}},\dots,2\pi
t_{p}J_{m_{p}})=\underset{l=1}{\overset{s}{\sum}}2\pi t_{l}L_{l}$
 and therefore $M=\underset{j=1}{\overset{p}{\sum}}k_{j}B_{j}+\underset{l=1}{\overset{s}{\sum}}2\pi t_{l}L_{l}.$
We conclude that $\mathrm{g}=
\underset{j=1}{\overset{p}{\sum}}\mathbb{Z}B_{j}+2\pi\underset{l=1}{\overset{s}{\sum}}\mathbb{Z}L_{l}.$
\
\\
$\bullet$  Second, we determine $\mathrm{g}_{u_{0}}$. Let $B\in \mathrm{g}$. We have $B=
\underset{j=1}{\overset{p}{\sum}}k_{j}B_{j}+2\pi\underset{l=1}{\overset{s}{\sum}}t_{l}L_{l}$
 for some $k_{1},\dots,k_{p}\in\mathbb{Z}$, and $t_{1},\dots,t_{s}\in\mathbb{Z}$. As $\widetilde{L_{l,i}}f_{i,1} = f^{(l)}_{i}$, 
 $i = 1,\dots,s$ then

\begin{align*}
L_{l}u_{0}& =\mathrm{diag}(0,\dots,0;\ \widetilde{L_{l,1}},\dots,\widetilde{L_{l,s}})[e_{1,1},\dots,e_{r,1};\ f_{1,1},\dots,f_{s,1}]^{T}\\
\ &  = [
0,\dots,0;f^{(l)}_{1},\dots, f^{(l)}_{s}]^{T}\\
\ & =f^{(l)}.
 \end{align*}
 Hence $Bu_{0}=\underset{j=1}{\overset{p}{\sum}}k_{j}B_{j}u_{0}+2\pi\underset{l=1}{\overset{r}{\sum}}t_{l}f^{(l)}$
and therefore $\mathrm{g}_{u_{0}}=
\underset{j=1}{\overset{p}{\sum}}\mathbb{Z}(B_{j}u_{0})+2\pi\underset{l=1}{\overset{s}{\sum}}\mathbb{Z}f^{(l)}.$
This proves the Proposition.
\end{proof}
\bigskip

\begin{lem}\label{L:14} Let  $G$  be an abelian subgroup of  $\mathcal{K}_{\eta,r,s}^{*}(\mathbb{R})$.
Then:
$$\overset{\circ}{\overline{G(u_{0})}}\neq \emptyset\ \ \ if \ and \ only\ if \ \ \overset{\circ}{\overline{G^{2}(u_{0})}}\neq
 \emptyset.$$
\end{lem}

\begin{proof}
(i) Suppose that
$\overset{\circ}{\overline{G^{2}(u_{0})}}\neq \emptyset$. Since
$G^{2}(u_{0})\subset G(u_{0})$, it follows that
$\overset{\circ}{\overline{G(u_{0})}}\neq \emptyset$. Conversely, suppose that  $\overset{\circ}{\overline{G(u_{0})}}\neq
\emptyset$. Then by Theorem
~\ref{T:1}, $\overline{\mathrm{g}_{u_{0}}} = \mathbb{R}^{n}$. As $\mathrm{g}\subset \dfrac{1}{2}\mathrm{g}^{2}$
(since if $B\in \mathrm{g}$, we have $e^{2B}= (e^{B})^{2}\in G^{2}$), then
 $\overline{\frac{1}{2}\mathrm{g}^{2}_{u_{0}}} = \mathbb{R}^{n}$ and so  $\overline{\mathrm{g}^{2}_{u_{0}}} = \mathbb{R}^{n}$.
By applying Theorem
~\ref{T:1} to the abelian subgroup $G^{2}$, it follows that $\overset{\circ}{\overline{G^{2}(u_{0})}}\neq\emptyset$.
\end{proof}
\medskip

\begin{cor}\label{C:45} Let  $G$  be an abelian subgroup of  $\mathcal{K}^{*}_{\eta,r,s}(\mathbb{R})$. Then $G$ has
 a locally dense orbit if and only if so is $G^{2}$.
\end{cor}
\medskip

\begin{proof} This is a consequence from Lemmas \ref{L:101} and \ref{L:14}.
\end{proof}
\
\\
{ \it Proof of Theorem  ~\ref{T:2}.} One can assume by Proposition ~\ref{p:1} that $G$ is an abelian subgroup of
$\mathcal{K}^{*}_{\eta,r,s}(\mathbb{R})$. By applying Theorem ~\ref{T:1} to the subgroup $G^{2}$ of $\mathcal{K}^{+}_{\eta,r,s}(\mathbb{R})$, then
$(ii) \Leftrightarrow (iii)$ follows from Proposition ~\ref{p:11} and Lemma \ref{L:14}. $(i) \Leftrightarrow (ii)$ follows
from Theorem \ref{T:1}.
\medskip

Recall the following proposition which was proven in \cite{mW}:

\begin{prop}\label{p:12}$($\cite{mW}, Proposition 4.3$)$ Let  $H = \mathbb{Z}u_{1} + \dots + \mathbb{Z}u_{m}$
with  $u_{k } = (u_{k,1},\dots,u_{k,n})\in\mathbb{R}^{n}$,
  $k=1,\dots,m$. Then  $H$  is dense in  $\mathbb{R}^{n}$  if and only if for every  $(s_{1},\dots,s_{m})\in
 \mathbb{Z}^{m}\backslash \{0\}$:
$$\mathrm{rank}\left(\left[\begin{array}{cccc }
 u_{1,1 } &\dots &\dots. & u_{m,1 } \\
 \vdots &\vdots &\vdots &\vdots \\
 u_{1, n } &\dots &\dots & u_{m,n } \\
 s_{1 } &\dots &\dots & s_{m }
 \end{array}\right]\right) = n+1.$$
\end{prop}
\
\\
\\
{\it Proof of Corollary  ~\ref{CC:2}.} $(i)\Longrightarrow (iii)$ follows from Corollary \ref{C:1} and Theorem \ref{T:2}.
$(iii)\Longrightarrow (i)$ follows from Corollary \ref{C:1}. $(ii)\Longleftrightarrow (iii)$ follows from
Proposition \ref{p:12}.\qed
\
\\
\\
{\it Proof of Corollary  ~\ref{C:3}.} Let $A_{1},\dots,A_{p}$ generate $G$ and let  $B_{1},\dots,B_{p}\in \mathrm{g}$
so that $A_{1}^{2}= e^{B_{1}},\dots, A_{p}^{2} = e^{B_{p}}$. Set
$v_{k} = (B_{k}v_{0}; s_{k})$  and  $w_{l}= (2\pi
Pf^{l}; t_{l})$, $1\leq k\leq p$, $1\leq l\leq s$. If  $p\leq n-s$  then
$\mathrm{rank}(v_{1},\dots,v_{p}; \ w_{1},\dots,w_{s})\leq n$  and
hence by Proposition \ref{p:12}, $\mathrm{g}^{2}_{v_{0}}=\underset{k=1}{\overset{p}{\sum}}\mathbb{Z}(B_{k}v_{0}) +
\underset{l=1}{\overset{s}{\sum}}2\pi\mathbb{Z}Pf^{(l)}$ is not dense in $\mathbb{R}^{n}$. Hence, by Theorem ~\ref{T:2}, $G$
has nowhere dense orbit, in particular, it is not topologically transitive.\qed
\medskip
\
\\
\\
 {\it Proof of Corollary  ~\ref{C:4}.}  Since \ $r+2s\leq n$, it follows that
 $p+s\leq \left[\frac{n+1}{2}\right]+ \frac{n-r}{2}\leq n+ \frac{1-r}{2}\leq n+\frac{1}{2}$. Hence,
$p+s\leq n$ and therefore Corollary ~\ref{C:4} follows from Corollary ~\ref{C:3}. \ \qed

\section{The case $n=2$ and some examples}
\medskip

For a given partition $\eta=(n_{1},\dots,n_{r})$ of $n$, we see that $r,s\in\{0,1,2\}$ and $n_{i}\in\{0,1,2\}$.
 In this case, we have  $\mathcal{K}^{*}_{1,1,0}(\mathbb{R}) =
\mathbb{T}^{*}_{2}(\mathbb{R})$,  $\mathcal{K}^{*}_{(1,1),2,0}(\mathbb{R})
= \mathbb{D}^{*}_{2}(\mathbb{R})$  and
$\mathcal{K}^{*}_{1,0,1}(\mathbb{R}) = \mathbb{B}^{*}_{1}(\mathbb{R})=\mathbb{S}^{*}$  where

\begin{align*}
  \ & \mathbb{D}^{*}_{2}(\mathbb{R})= \left\{\left[\begin{array}{cc }
   a & 0 \\
   0 & b
   \end{array}\right]: \ a, b\in \mathbb{R}^{*}\right\}, \\
\\
  \ & \mathbb{T}^{*}_{2}(\mathbb{R})= \left\{\left[\begin{array}{cc }
   a & 0 \\
   b & a
   \end{array}\right]: \ a,b\in \mathbb{R}, a\neq0 \ \right\} \\
   \\
\mathrm{and}\ \ \ \  \ & \ \\
\ & \mathbb{S}^{*} = \left\{\left[\begin{array}{cc }
   \alpha & -\beta\\
   \beta & \alpha
   \end{array}\right]: \ \alpha, \beta\in \mathbb{R}, \ \alpha^{2}+\beta^{2}\neq0\right\}.
\end{align*}
\
\\
Note that $\mathbb{D}^{*}_{2}(\mathbb{R})$, $\mathbb{T}^{*}_{2}(\mathbb{R})$ and $\mathbb{S}^{*}$ are all abelian.\\
Let  $G$  be a subgroup of
$\mathcal{K}^{*}_{\eta,r,s}(\mathbb{R})$, $r,s= 0,1,2$. We distinguish three cases:
\
\\
\\
\emph{Case 1:  $G$  is a subgroup of
$\mathbb{D}^{*}_{2}(\mathbb{R})$.} Then we have the following proposition.
\smallskip

\begin{prop}\label{p:13} Let  $A_{k} =\mathrm{diag}( \lambda_{k}, \mu_{k} )$,  where  $\lambda_{k}, \
\mu_{k} \in \mathbb{R}^{*}$,  $k=1,\dots,p$ and  $G$  be the group that they
generate. Then  $G$  has a dense
orbit if and only if $\textrm{ind}(G)=2$ and for every  $(s_{1},\dots,s_{p
})\in \mathbb{Z}^{p}\backslash\{0\}$:
$$\mathrm{rank}\left(\left[\begin{array}{ccc}
  2\log|\lambda_{1}| & \dots & 2\log|\lambda_{p}|  \\
  2\log|\mu_{1}| & \dots & 2\log|\mu_{p}|  \\
  s_{1} & \dots & s_{p}  \\
\end{array}\right]\right)=3.$$
\end{prop}
\medskip

\begin{proof} We let  $B_{k} = \mathrm{diag}(
 2\log|\lambda_{k}|, 2\log|\mu_{k}|), \ k=1,\dots,p.$ \ One has  $e^{B_{k}} = A_{k}^{2}$  and
 $B_{k}\in \mathbb{D}^{*}_{2}(\mathbb{R})$, $k=1,\dots,p$. Then  $B_{k}\in \mathrm{g}$ and by Corollary ~\ref{CC:2},
  the proposition follows.
\end{proof}
\medskip

\begin{exe} Let $G$  be the group generated by:

$A_{1} = \mathrm{diag}( -e^{\frac{\sqrt{2}}{2}},1), \ A_{2} =
\mathrm{diag}(
 1, -e^{\frac{1}{2}})$  and  $A_{3} = \mathrm{diag}(e^{-\frac{\sqrt{3}}{2}}, e^{-\frac{\sqrt{2}}{2}})$.
\\
Then every orbit in  $\mathbb{R}^{*}\times \mathbb{R}^{*}$  is
dense in $\mathbb{R}^{2}$.
\end{exe}
\medskip

\begin{proof} We see that $\textrm{ind}(G)=2$ and $U=\mathbb{R}^{*}\times \mathbb{R}^{*}$.
 Moreover, for every  $(s_{1}, s_{2}, s_{3})\in\mathbb{Z}^{3}\backslash\{0\}$,
one has the determinant: $$\Delta =
\mathrm{det}\left[\begin{array}{ccc}
                 \sqrt{2} & 0 & -\sqrt{3} \\
                 0 & 1 &  -\sqrt{2} \\
                 s_{1} & s_{2} & s_{3}
               \end{array}\right]=s_{1}\sqrt{3}+ 2s_{2}+ s_{3}\sqrt{2}$$

 Since  $2$, $\sqrt{2}$  and  $\sqrt{3}$  are rationally independent, $\Delta\neq 0$.  Therefore:
$$\mathrm{rank}\left(\left[\begin{array}{ccc}
                 \sqrt{2} & 0 & -\sqrt{3} \\
                 0 & 1 &  -\sqrt{2} \\
                 s_{1} & s_{2} & s_{3}
               \end{array}\right]\right) = 3$$
 and by Proposition ~\ref{p:11},  $G$  has a dense orbit.  We conclude by
Corollary ~\ref{C:6}, that every orbit in $\mathbb{R}^{*}\times\mathbb{R}^{*}$  is  dense in $\mathbb{R}^{2}$.
\end{proof}
\
\\
\emph{ Case 2:  $G$  is a subgroup of
$\mathbb{T}_{2}^{*}(\mathbb{R})$.} Then we have the following proposition.

\begin{prop}\label{p:14} Let  $A_{k} = \left[\begin{array}{cc }
  \lambda_{k } & 0 \\
  \mu_{k } & \lambda_{k }
  \end{array}\right]$  where  $\lambda_{k}\in \mathbb{R}^{*}$, $\mu_{k}\in \mathbb{R}$,  $k=1,\dots, p$ and
    $G$  be the group that they generate. Then  $G$ has a dense orbit if and only if $\textrm{ind}(G)=1$ and for every
$(s_{1},\dots,s_{p})\in \mathbb{Z}^{p}\backslash \{0\}$:

$$\mathrm{rank}\left(\left[\begin{array}{ccc}
  2\log|\lambda_{1}| & \dots & 2\log|\lambda_{p }|  \\
   2\frac{\mu_{1}}{\lambda_{1}} & \dots  & 2\frac{\mu_{p}}{\lambda_{p}} \\
   s_{1} & \dots  & s_{p}   \\
\end{array}\right]\right)=3.$$
\end{prop}
\bigskip

\begin{proof} We let  $B_{k} = \left[\begin{array}{cc}
  2\mathrm{log}|\lambda_{k }|& 0 \\
  \\
  \frac{2\mu_{k}}{\lambda_{k}} & 2\log|\lambda_{k }| \
  \end{array}\right]\in \mathbb{T}_{2}(\mathbb{R})$,  $k=1,\dots,p$. One has  $e^{B_{k}} = A_{k}^{2}$
and  $B_{k}\in \mathbb{T}_{2}(\mathbb{R})$.
  Then  $B_{k}\in \mathrm{g}$. The proposition follows then from  Corollary ~\ref{CC:2}.
\end{proof}
\bigskip

\begin{exe} Let  $G$  be the group generated by:

$A_{1} = \left[\begin{array}{cc }
 -e^{\sqrt{2}} & 0 \\
 0 &  -e^{\sqrt{2}}
\end{array}\right], \ \ A_{2} = \left[\begin{array}{cc}
 1 & 0 \\
 1 & 1 \end{array}\right] \ \mathrm{and} \ A_{3} = \left[\begin{array}{cc}
 e^{-\sqrt{3}} & 0\\
 -\sqrt{2}e^{-\sqrt{3}}& e^{-\sqrt{3}}
\end{array}\right]$.
\\
Then every orbit in  $\mathbb{R}^{*}\times \mathbb{R}$  is dense
in $\mathbb{R}^{2}$.
\end{exe}
\bigskip

\begin{proof} We see that $\textrm{ind}(G)=1$ since $r=1$ and $-e^{\sqrt{2}}$ is an eigenvalue of $A_{1}$,
$U=\mathbb{R}^{*}\times \mathbb{R}$ and for every  $(s_{1}, s_{2},
s_{3})\in\mathbb{Z}^{3}\backslash\{0\}$:
$$\mathrm{rank}\left(\left[\begin{array}{ccc}
                 2\sqrt{2} & 0 & -2\sqrt{3} \\
                 0 & 2 &  -2\sqrt{2} \\
                 s_{1} & s_{2} & s_{3}
               \end{array}\right]\right)=3$$
So  $G$ has a dense orbit. By Corollary ~\ref{C:6}, every orbit in
$\mathbb{R}^{\ast}\times\mathbb{R}$  is dense in $\mathbb{R}^{2}$.
\end{proof}
\
\\
\emph{Case 3:  $G$  is a subgroup of  $\mathbb{S}^{\ast}$.} Then we have the following proposition.
\smallskip

 \begin{prop}\label{p:15} Let  $A_{k} = |\lambda_{k
 }|\left[\begin{array}{cc }
  \cos \theta_{k} & -\sin\theta_{k} \\
 \sin \theta_{k} & \cos \theta_{k}
  \end{array}\right]$  where  $\lambda_{k}\in \mathbb{R}^{*}$, $\theta_{k}\in \mathbb{R}$, \ $k=1,..., p$
  and $G$  the group that they generate. Then  $G$  has a dense orbit if and only if
for every  $(s_{1},\dots,s_{p})\in \mathbb{Z}^{p+1}\backslash\{0\}$:
$$\mathrm{rank}\left(\left[\begin{array}{cccc}
  2\log|\lambda_{1}| & \dots & 2\mathrm{log}|\lambda_{p}|& 0 \\
   2\theta_{1} & \dots  & 2\theta_{p}& 2\pi  \\
   s_{1} & \dots  & s_{p}& t   \\
\end{array}\right]\right)=3.$$
\end{prop}
\medskip

  \begin{proof} We see that $G\subset \mathbb{B}^{*}_{1}(\mathbb{R})=\mathbb{S}^{*}$ and so  ind($G)=r=0$. We let  $B_{k} = \left[\begin{array}{cc}
    2\log|\lambda_{k }| &-2\theta_{k} \\
   2\theta_{k} & 2\log|\lambda_{k }| \\
  \end{array}\right]$. For every $k=1,\dots,p$, one has  $e^{B_{k}} = A^{2}_{k}$  and  $B_{k}\in \mathbb{S}$. By Corollary ~\ref{CC:2},
Proposition ~\ref{p:15} follows.
\end{proof}
\medskip

\begin{exe}\label{e:1} Let  $G$  be the group generated by:

$$A_{1}=e^{\sqrt{2}}\left[\begin{array}{cc }
 \cos\left(-\sqrt{3}\right) & -\sin\left(-\sqrt{3}\right) \\
 \sin\left(-\sqrt{3}\right)&  \cos\left(- \sqrt{3}\right)
\end{array}\right]\ \ \ \mathrm{and} \ \ A_{2}=e^{\sqrt{3}}\left[\begin{array}{cc }
\cos \sqrt{2} & -\sin \sqrt{2} \\
 \sin \sqrt{2}&  \cos \sqrt{2}
  \end{array}\right].$$\\
Then every orbit in $\mathbb{R}^{2}\backslash\{0\}$ is dense in
$\mathbb{R}^{2}$.
\end{exe}
\smallskip

\begin{proof} We see that $G\subset \mathbb{S}^{*}$ and so  ind($G)=r=0$. For every
$(s_{1}, s_{2}, t)\in\mathbb{Z}^{3}\backslash\{0\}$, one has the determinant:
$$\mathrm{det}\left[\begin{array}{ccc}
                 2\sqrt{2} &  2\sqrt{3} & 0 \\
                 -2\sqrt{3} &  -2\sqrt{2} & 2\pi \\
                 s_{1} & s_{2} & t
               \end{array}\right]= 4\pi(\sqrt{3}s_{1}-\sqrt{2}s_{2})+4t\neq 0.$$
It follows that: $$\mathrm{rank}\left(\left[\begin{array}{ccc}
                  2\sqrt{2} &  2\sqrt{3} & 0 \\
                 -2\sqrt{3} &  -2\sqrt{2} & 2\pi \\
                 s_{1} & s_{2} & t
               \end{array}\right]\right)=3.$$
By  Corollary ~\ref{CC:2},  $G$  has a dense orbit. Since  $U =
\mathbb{R}^{2}\backslash\{0\}$, it follows by Corollary ~\ref{C:6} that every orbit in  $\mathbb{R}^{2}\backslash \{0\}$ \
is dense in $\mathbb{R}^{2}$.
\end{proof}
\smallskip

\begin{cor}\label{C:8} If $G$ is a finitely generated abelian subgroup of SL($2,\mathbb{R})$, it is not topologically transitive.
\end{cor}
\smallskip

\begin{proof} One can assume  by Proposition ~\ref{p:1} that $G$ is a subgroup of $\mathbb{D}^{*}_{2}(\mathbb{R})$,
 $\mathbb{T}^{*}_{2}(\mathbb{R})$ or  $\mathbb{S}^{*}$. Let $A_{1},\dots, A_{p}$ generate $G$.
\
\\
- If $G\subset \mathbb{D}^{*}_{2}(\mathbb{R})$, then, one can write
$A_{k} =\mathrm{diag}(\lambda_{k}, \frac{1}{\lambda_{k}})$,
where  $\lambda_{k} \in \mathbb{R}^{*}$, $k=1,\dots,p$. Then $\log\left|\frac{1}{\lambda_{k}}\right|=-\log|\lambda_{k}|$ and
by Proposition ~\ref{p:13},  $G$  has no dense orbit.
\
\\
- If $G\subset \mathbb{T}^{*}_{2}(\mathbb{R})$, then one can write  $A_{k} =\left[\begin{array}{cc}
                                                                                                           \lambda_{k} & 0 \\
                                                                                                           \mu_{k} & \lambda_{k}
                                                                                                         \end{array}
\right]$,
where  $\mu_{k} \in \mathbb{R}$ and $|\lambda_{k}|=1$, $k=1,\dots,p$. Then $\log|\lambda_{k}|=0$ and by
Proposition ~\ref{p:14},  $G$  has no dense orbit.
\
\\
- If $G\subset \mathbb{S}^{*}$, then one can write  $$A_{k} =|\lambda_{k}|\left[\begin{array}{cc}
                                                                                                           \cos\theta_{k} & -\sin\theta_{k}\\
                                                                                                           \sin\theta_{k} & \cos\theta_{k}
                                                                                                         \end{array}
\right],$$
where  $\theta_{k} \in \mathbb{R}$ and $|\lambda_{k}|=1$, $k=1,\dots,p$. Then $\log|\lambda_{k}|=0$ and
by Proposition ~\ref{p:15},  $G$  has no dense orbit.
\end{proof}

\textbf{Remark 2}. We proved in (\cite{aAhM06}, Corollary 1.5), that no abelian subgroup of  $\textrm{GL}(n,
\mathbb{C})$  generated by $n$ \ matrices $(n\geq 1$) is topological transitive. Example ~\ref{e:1}
 shows that this property is not true for abelian subgroup of
 $\mathrm{GL}(2, \mathbb{R})$. The following is another example for $n=4$:

\begin{exe} Let  $G$  be the abelian group generated by:

\begin{align*}
  A_{1} &  =\left[\begin{array}{cc }
 \left[\begin{array}{cc}
          e^{\frac{\sqrt{3}}{4\pi}}\cos(\frac{1}{2}) & -e^{\frac{\sqrt{3}}{4\pi}}\sin(\frac{1}{2}) \\
           e^{\frac{\sqrt{3}}{4\pi}}\sin(\frac{1}{2})&  e^{\frac{\sqrt{3}}{4\pi}}\cos(\frac{1}{2})
       \end{array}\right] & 0 \\
0&  \left[\begin{array}{cc}
          e^{\frac{\sqrt{5}}{2}}\cos(\frac{\sqrt{2}}{2}) & -e^{\frac{\sqrt{5}}{2}}\sin(\frac{\sqrt{2}}{2})  \\
           e^{\frac{\sqrt{5}}{2}}\sin(\frac{\sqrt{2}}{2}) &  e^{\frac{\sqrt{5}}{2}}\cos(\frac{\sqrt{2}}{2})
       \end{array}\right]
\end{array}\right] \\
\\
  A_{2}&  = \textrm{diag}(\sqrt{e} I_{2}, I_{2}) \ \  \ \mathrm{and} \ \ \  A_{3} = \textrm{diag}(I_{2}, \sqrt{e} I_{2}).
\end{align*}

Then every orbit in  $(\mathbb{R}^{2}\backslash \{0\})^{2}$  is dense in $\mathbb{R}^{4}$.
\end{exe}
\smallskip

\begin{proof} We see that $G\subset \mathcal{K}^{*}_{\eta,0,2}(\mathbb{R})=\mathbb{B}^{*}_{1}(\mathbb{R})\oplus
\mathbb{B}^{*}_{1}(\mathbb{R})$ where $\eta=(2,2)$. Hence $\textrm{ind}(G)=0$ and $U = (\mathbb{R}^{2}\backslash\{0\})^{2}$.  Write

\begin{align*}
  B_{1} &= \left[\begin{array}{cc}
 \left[\begin{array}{cc}
          \frac{\sqrt{3}}{2\pi} & -1 \\
          1 & \frac{\sqrt{3}}{2\pi}
       \end{array}\right] &  0\\
0&  \left[\begin{array}{cc}
          \sqrt{5} & -\sqrt{2} \\
           \sqrt{2} & \sqrt{5}
       \end{array}\right]
\end{array}\right] \\
\\
 B_{2}&  = \textrm{diag}(I_{2},\ 0_{2}) \ \ \ \ \ \mathrm{and}\ \ \ B_{3} = \textrm{diag}(0_{2}, I_{2})
\end{align*}

where $0_{2}=\left[\begin{array}{cc}
                     0 & 0 \\
                     0 & 0
                   \end{array}
\right]\in M_{2}(\mathbb{R})$.

One has  $e^{B_{k}} = A_{k}^{2}$, \ $k=1,2,3$. For every  $(s_{1}, s_{2},s_{3}, t_{1},
t_{2})\in\mathbb{Z}^{5}\backslash \{0\}$, we have
\begin{align*}
\Delta & :=\left[\begin{array}{ccccc}
                   \frac{\sqrt{3}}{2\pi} & 1 & 0 & 0 & 0 \\
                   1 & 0 & 0 & 2\pi & 0 \\
                  \sqrt{5} & 0 & 1& 0 & 0 \\
                  \sqrt{2}  & 0 & 0 & 0 & 2\pi \\
                  s_{1} & s_{2} & s_{3} & t_{1} & t_{2}
                \end{array}
\right]\\
\ &\ =\pi(\pi(-4s_{1}+4\sqrt{5}s_{3})+2\sqrt{3}s_{2}+2t_{1}+2\sqrt{2} t_{2})
\end{align*}
Since $\pi$ is a transcendent number, $\Delta \neq 0$.
It follows that $$\mathrm{rank}\left(\left[\begin{array}{ccccc}
                   \frac{\sqrt{3}}{2\pi} & 1 & 0 & 0 & 0 \\
                   1 & 0 & 0 & 2\pi & 0 \\
                  \sqrt{5} & 0 & 1& 0 & 0 \\
                  \sqrt{2}  & 0 & 0 & 0 & 2\pi \\
                  s_{1} & s_{2} & s_{3} & t_{1} & t_{2}
                \end{array}
\right]\right) = 5.$$  By  Corollary ~\ref{CC:2},  $G$  has a dense orbit in
$\mathbb{R}^{4}$. By Corollary ~\ref{C:6},
every orbit in  $(\mathbb{R}^{2}\backslash\{0\})^{2}$  is dense in
$\mathbb{R}^{4}$.
\end{proof}
\bigskip

\section{{\bf Appendix}}
\
\\
{\it Proof of Lemma ~\ref{L:2}.}\ $\bullet$ Let's show first that  $$\textrm{Ker}(e^{B}-e^{\mu } I_{n})\subset
\textrm{Ker}(e^{tB}-e^{t\mu } I_{n})\ \mathrm{for} \
\mathrm{all} \ t\in \mathbb{R}$$
\
\\
 Let $v\in \textrm{Ker}(e^{B}-e^{\mu }I_{n})$. Then
for all $m\in \mathbb{N}$, one has $$e^{mB}v=e^{m\mu}v \ \ \ \ \ \
(1).$$
\\
For all $i=1,\dots, n$, denote by $P_{i}(t)=<e^{-t\mu}(e^{tB}-e^{t\mu}I_{n})v,\ e_{i}>,$
 where $<,\ >$ is the scalar product on $\mathbb{R}^{n}$. Write  $N=B-\mu I_{n}$. Then $N$ is nilpotent and then  for all $t\in \mathbb{R}$,
\\
$$e^{-t\mu}(e^{tB}-e^{t\mu}I_{n})=e^{tN}-I_{n}=\underset{k=1}{\overset{k=n}{\sum}}\frac{t^{k}N^{k}}{k!}.$$
It follows that $P_{i}$  is a polynomial of degree at most $n$.
\
\\
\\
According to $(1)$, one has $P_{i}(m)=0$   for all $m\in
\mathbb{N}$ and $i=1.., n$. Therefore
$(e^{tB}-e^{t\mu}I_{n})v=0$, for all $t\in\mathbb{R}$ and so
$v\in \textrm{Ker}(e^{tB}-e^{t\mu } I_{n})$, for all $t\in \mathbb{R}$.
\
\\
\\
$\bullet$  {\it Proof of $(i)$:} Let's prove that $\textrm{Ker}(B-\mu I_{n})=\textrm{Ker}(e^{B}-e^{\mu}I_{n})$
\\
For any $v\in \textrm{Ker}(e^{B}-e^{\mu}I_{n})$, denote by $\varphi_{v }:  \ t \ \rightarrow \
(e^{tB}-e^{t\mu}I_{n})v$, one has $\varphi_{v}(t)=0$ for all
$t\in\mathbb{R}$. Thus $\frac{\partial \varphi_{v}}{\partial t}(t)=(Be^{tB}-\mu e^{t
\mu}I_{n})v=0$ for all $t\in\mathbb{R}$. In particular for $t=0$,
 $(B-\mu I_{n})v=0$ and so  $v\in \textrm{Ker}(B-\mu I_{n})$, this proves that $$Ker(e^{B}-e^{\mu}I_{n })\subset \textrm{Ker}(B-\mu I_{n}) \ \
\ (2).$$
\
\\
Conversely, let $v\in \textrm{Ker}(B-\mu I_{n})$. Then $Bv=\mu v$ and so
$$e^{B}v=\underset{k=0}{\overset{+\infty}{\sum}}\frac{B^{k}}{k!}v
=
\underset{k=0}{\overset{+\infty}{\sum}}\frac{\mu^{k}}{k!}v=e^{\mu}v.$$
\\
thus $v\in \textrm{Ker}(e^{B}-e^{\mu}I_{n})$, this proves that
$$Ker(B-\mu I_{n})\subset \textrm{Ker}(e^{B}-e^{\mu}I_{n }) \ \ \ (3).$$
It follows from $(2)$ and $(3)$,  that $\textrm{Ker}(B-\mu I_{n }) =
\textrm{Ker}(e^{B}-e^{\mu}I_{n})$.
\
\\
\\
$\bullet$ {\it Proof of $(ii)$:} The proof is done by induction on $n$.  For $n=1$, it is obvious. Suppose
 that $(ii)$ is true until the order $n-1$ and let $B\in M_{n}(\mathbb{R})$ having only one eigenvalue
$\mu$ so that $e^{B}\in\mathbb{T}^{+}_{n}(\mathbb{R})$. By $(i)$, one has  $$\textrm{Ker}(B-\mu I_{n})=\textrm{Ker}(e^{B}-e^{\mu}I_{n })$$
\\
 Since $e_{n}\in \textrm{Ker}(e^{B}-e^{\mu}I_{n})$, so $e_{n}\in \textrm{Ker}(B-\mu I_{n})$ and one can write
$$ B=\left[\begin{array}{cc }
 B^{(1) } & 0\\
 L_{B } & \mu
\end{array}\right] \ \ \mathrm{and } \ \ \ \ e^{B}=\left[\begin{array}{cc } (e^{B})^{(1) } & 0 \\ L_{e^{B } }
& e^{\mu }
\end{array}\right] \ $$
\\
with $ (e^{B})^{(1)}\in M_{n-1}(\mathbb{R})$ and $B^{(1)}\in
M_{n-1}(\mathbb{R})$.
\
\\
\\
Since $e^{B}\in \mathbb{T}^{+}_{n}(\mathbb{R})$, it follows that
$(e^{B})^{(1)}\in \mathbb{T}^{+}_{n-1}(\mathbb{R})$.  By the induction hypothesis,  $B^{(1)}\in
\mathbb{T}_{n-1}(\mathbb{R})$ and so $B\in\mathbb{T}_{n}(\mathbb{R})$. \qed
\\
\medskip

 \textit{Proof of Proposition ~\ref{p:2}}: The proof is a consequence of the following results.
\
\begin{itemize}
  \item [(i)] $\textrm{exp}(\mathbb{T}_{n}(\mathbb{R}))=\mathbb{T}^{+}_{n}(\mathbb{R})$.
  \item [(ii)] $\textrm{exp}(\mathbb{T}_{n}(\mathbb{C}))=\mathbb{T}^{\ast}_{n}(\mathbb{C})$.
  \item [(iii)] $\textrm{exp}(\mathbb{B}_{m}(\mathbb{R}))=\mathbb{B}^{\ast}_{m}(\mathbb{R})$ for every $m\in\mathbb{N}_{0}$.
\end{itemize}
\
\\
{\it Proof of  (i)}: Let $A\in \mathbb{T}_{n}(\mathbb{R})$ with eigenvalue $\lambda$.
Then $e^{A}\in \mathbb{T}_{n}(\mathbb{R})$ with eigenvalue
$e^{\lambda}> 0$. Hence $e^{A}\in \mathbb{T}^{+}_{n}(\mathbb{R})$ and
therefore $\textrm{exp}(\mathbb{T}_{n}(\mathbb{R}))\subset\mathbb{T}^{+}_{n}(\mathbb{R})$.
\
\\
Conversely, let $A\in \mathbb{T}^{+}_{n}(\mathbb{R})$ and
$J=\mathrm{diag}(J_{1},\dots,
J_{r})\in\mathbb{T}^{+}_{n}(\mathbb{R})$ its reduced Jordan form
  where $J_{k } = \left[\begin{array}{ccccc }
 \lambda & \ &\ &\ & 0 \\
 1 &\ddots &\ &\ & \ \\
0 &\ddots &\ddots & \  & \ \\
\vdots &\ddots &\ddots & \ddots  & \ \\
 0 & \dots & 0 &1& \lambda
\end{array}\right]\in \mathbb{T}^{+}_{n_{k}}(\mathbb{R})$ if
$n_{k}\geq2$ and $J_{k}=(\lambda)$ if $n_{k}=1$.
\
\\
Since $\lambda> 0$, there exists $\mu\in\mathbb{R}$, such that
$e^{\mu}=\lambda$. Denote by $$J^{\prime}=\mathrm{diag}( J^{\prime}_{1 },\dots,J^{\prime} _{r })\in
\mathbb{T}^{+}_{n}(\mathbb{R})$$ where

$J^{\prime}_{k } = \left[\begin{array}{ccccc }
 \mu & \ &\ &\ & 0 \\
 1 &\ddots &\ &\ & \ \\
0 &\ddots &\ddots & \  & \ \\
\vdots &\ddots &\ddots & \ddots  & \ \\
 0 & \dots & 0 &1& \mu
\end{array}\right]\in\mathbb{T}_{n_{k}}(\mathbb{R})$ if $n_{k}\geq 2$ and $J_{k}=(\mu)$
if $n_{k}=1$.  Then $e^{J'}=\mathrm{diag}( e^{J^{\prime}_{1}},\dots, e^{J^{\prime}_{r}})$ with
$$e^{J^{\prime}_{k}}=e^{\mu}\left[\begin{array}{cccccc }
 1 & \ &\ &\ &\ & 0 \\
 1 &\ddots &\ &\ &\ & \ \\
 \frac{1}{2 } &\ddots &\ddots &\  &\ & \ \\
 \vdots &\ddots &\ddots &\ddots  &\  & \ \\
 \vdots &\ddots &\ddots &\ddots &\ddots & \ \\
\frac{1}{(n_{k}-2)! }  &\dots &\dots & \frac{1}{2}&1& 1 \\
\end{array}\right]$$
 \
 \\
 One can check that $\textrm{dim}
\left(\textrm{Ker}(e^{J^{\prime}_{k}}-e^{\mu}I_{n_{k}})\right)=1$, which gives
that the reduced Jordan form associated to $e^{J^{\prime}_{k}}$ is
  $$\left[\begin{array}{ccccc }
 e^{\mu} & \ &\ &\ & 0 \\
 1 &\ddots &\ &\ & \ \\
0 &\ddots &\ddots & \  & \ \\
\vdots &\ddots &\ddots & \ddots  & \ \\
 0 & \dots & 0 &1& e^{\mu}
\end{array}\right]=J_{k}, \ \ \ \mathrm{for \ all } \ k=1,\dots,
r.$$
\\
 Thus there exists $P_{k}\in \textrm{GL}(n_{k }, \mathbb{R})$ such that $P_{k}e^{J^{\prime}_{k}}P_{k}^{-1}=J_{k}$, for
all $k=1,\dots, r$. Denote by $P= \mathrm{diag}( P_{1},\dots, P_{r})$, one has $Pe^{J^{\prime}}P^{-1}=J$.
\\
 As $J$ is the reduced Jordan form associated to $A$ then there exists $Q\in \textrm{GL}(n,
\mathbb{R})$ such that $QAQ^{-1}=J$. It follows that  $e^{B}=A$ where $B=Q^{-1}PJ^{\prime}P^{-1}Q$
having only eigenvalue $\mu$. By Lemma ~\ref{L:2}, (ii), $B\in \mathbb{T}_{n}(\mathbb{R})$. This
proves the other inclusion $\mathbb{T}^{+}_{n}(\mathbb{R})\subset \textrm{exp}(\mathbb{T}_{n}(\mathbb{R}))$.
\medskip

{\it Proof of (ii)}: the proof is analogous to that of (i).
\medskip

{\it Proof of (iii)}:  By Lemma ~\ref{L:1}, there exists $Q\in \textrm{GL}(m,\mathbb{C})$ such that
$Q^{-1}\mathbb{B}_{m}(\mathbb{R})Q\subset\mathbb{T}_{m}(\mathbb{C})\oplus\mathbb{T}_{m}(\mathbb{C})$.\\
 So
$$\mathrm{exp}(\mathbb{B}_{m}(\mathbb{R}))\subset Q \left(\mathrm{exp}\left(\mathbb{T}_{m}(\mathbb{C})\right)
\oplus\mathrm{exp}\left(\mathbb{T}_{m}(\mathbb{C})\right)\right)Q^{-1}.$$ By (ii),
$\mathrm{exp}\left(\mathbb{T}_{m}(\mathbb{C})\right)
=\mathbb{T}^{*}_{m}(\mathbb{C})$,
so $$\mathrm{exp}(\mathbb{B}_{m}(\mathbb{R})) \subset Q\left(\mathbb{T}^{*}_{m}(\mathbb{C})\oplus \mathbb{T}^{*}_{m}(\mathbb{C})\right)Q^{-1}=\mathbb{B}^{*}_{m}(\mathbb{R}).$$
\
\\
Conversely, let $A\in \mathbb{B}^{*}_{m}(\mathbb{R})$, by Lemma ~\ref{L:1},  there exist $Q\in \textrm{GL}(m,\mathbb{C})$ such that
$Q^{-1}AQ=\mathrm{diag}(A_{1}, \overline{A_{1}})$, where $A_{1}\in \mathbb{T}^{*}_{m}(\mathbb{C})$. By (ii), there exists
 $B_{1}\in \mathbb{T}_{m}(\mathbb{C})$ such that $e^{B_{1}}=A_{1}$. Set $B=Q\mathrm{diag}(B_{1},\overline{B_{1}})Q^{-1}$. Then $B\in \mathbb{B}_{m}(\mathbb{R})$\\ and
 $$e^{B}=Q\mathrm{diag}(e^{B_{1}}, \overline{e^{B_{1}}})Q^{-1}= Q\mathrm{diag}(A_{1}, \overline{A_{1}})Q^{-1}=A.$$\
 This completes the proof.\qed
\\
\medskip
\medskip

\textit{Proof of Proposition ~\ref{p:3}}.  We need the following Lemma:
\medskip

\begin{lem}
 Let $A,B\in \mathbb{T}_{n}(\mathbb{K})$ $(\mathbb{K}=\mathbb{R}\ or \ \mathbb{C})$ so that $e^{A}e^{B}=e^{B}e^{A}$.
Then:
\begin{itemize}
  \item [(i)] $e^{tA}e^{tB}=e^{tB}e^{tA}$, for all \ $t\in\mathbb{R}$
  \item [(ii)] $AB=BA$
\end{itemize}
\end{lem}
\medskip

\begin{proof}
$\bullet$ {\it Proof of (i)}:

\textit{Step 1}. We prove by induction on $m$ that for all $m\in
\mathbb{N}^{*}$, one has $$e^{mA}e^{B}=e^{B}e^{mA} \ \ \ \ \ (\ast)$$
\
\\
For $m=1$, the formula $(\ast)$ is obvious by hypothesis. Supposed that $(\ast)$ is true for $m-1$
 and let $A, \ B\in \mathbb{T}_{n}(\mathbb{K})$ so that $e^{A}e^{B}=e^{B}e^{A}$. We have:
$$e^{mA}e^{B}=e^{(m-1)A}e^{A}e^{B}=e^{(m-1)A}e^{B}e^{A}=e^{B}e^{(m-1)A}e^{A}=e^{B}e^{mA}.$$
\
\\
 \textit{Step 2}.  We prove that for all $t\in\mathbb{R}$, one has
$e^{tN}e^{tM}=e^{tM}e^{tN}$, where $N=A-\lambda I_{n}$ and $M=B-\mu
I_{n}$,  $\lambda$ (resp. $\mu$) is the only eigenvalue of
$A$ (resp. $B$).
\
\\
Since $N$ and $M$ are nilpotent of order $n$, so for all $t\in
\mathbb{R}$, one has
$$e^{tN}=\underset{k=0}{\overset{n}{\sum}}\frac{t^{k}N^{k}}{k!}\ \ \
\mathrm{and} \ \ \ e^{tM}=\underset{k=0}{\overset{n}{\sum}}\frac{t^{k}M^{k}}{k!}.$$ Thus
\begin{align*}
e^{tN}e^{tM}-e^{tM}e^{tN} & \ =\underset{i=0}{\overset{n}{\sum}}\underset{j=0}{\overset{n}{\sum}}\frac{t^{i+j}N^{i}M^{j}}{(i!)(j!)}-\underset{i=0}{\overset{n}{\sum}}\underset{j=0}{\overset{n}{\sum}}\frac{t^{i+j}M^{i}N^{j}}{(i!)(j!)}\\
\ & \ =\underset{i=0}{\overset{n}{\sum}}\underset{j=0}{\overset{n}{\sum}}\frac{t^{i+j}(N^{i}M^{j}-M^{i}N^{j})}{(i!)(j!)}
\end{align*}
\\
For all $1\leq i, j \leq n$, the function $P_{i,j}: \mathbb{R}
 \longrightarrow \mathbb{K}$, defined by  $P_{i,j}(t)=<
(e^{tN}e^{tM}-e^{tM}e^{tN})e_{i},\ e_{j}>$  is a polynomial of
degree  $\leq 2n$, where $ <,\ > $ is the scalar product on $\mathbb{K}^{n}$. By  step 1, one has for all $1\leq i, j \leq n$ and all
$m\in\mathbb{N}$, $P_{i,j}(m)=0$,  thus for all $1\leq i, j \leq
n$, $P_{i,j}=0$. It follows that for all $t\in \mathbb{R}$,
$e^{tN}e^{tM}=e^{tM}e^{tN}$.
\
\\
\textit{ Step 3}. Since
$A=N+\lambda I_{n}$ and $B=M+\mu I_{n}$, we have
$e^{tA}=e^{t\lambda}e^{tN}$ and $e^{tB}=e^{t\mu}e^{tM}$, and by step 2, we have
 $e^{tA}e^{tB}=e^{tB}e^{tA}$. The proof is complete.
\
\\
\\
$\bullet$ {\it Proof of (ii)}: denote by $\Phi: \  \mathbb{R } \
\rightarrow \ \mathbb{T}^{*}_{n}(\mathbb{R})$ defined by
$\Phi(t)=e^{tA}e^{tB}e^{-tA}e^{-tB}$.
\\
According to $(i)$, one has for all $t\in\mathbb{R}$, $\Phi(t)=I_{n}$. In particular,
 the second derivative of $\Phi$ vanishes. We have

$$\frac{\partial
\Phi(t)}{\partial
t}=A+e^{tA}Be^{-tA}-e^{tB}Ae^{-tB}-B,
\ \ \ \forall t\in\mathbb{R }. $$

and
$$\frac{\partial^{2 }
\Phi(t)}{\partial
t^{2}}=Ae^{tA}Be^{-tA}-e^{tA}BAe^{-tA}-Be^{tB}Ae^{-tB}+e^{tB}ABe^{-tB},
\ \ \ \forall t\in\mathbb{R }. $$
Then for $t=0$, one  has  $\frac{\partial^{2}\Phi}{\partial t^{2}}(0)=2(AB-BA)=0$, hence  $AB=BA$.
\end{proof}
\medskip
\
\emph{Proof of Proposition ~\ref{p:3}}. Let $A, B \in \mathcal{K}_{\eta,r,s}(\mathbb{R})$ such that  $e^{A}e^{B}=e^{B}e^{A}$.
Write \ $A=\mathrm{diag}(A_{1},\dots,A_{r};\widetilde{A}_{1},\dots,\widetilde{A}_{s})$ \  and
\ $B=\mathrm{diag}(B_{1},\dots,B_{r};\widetilde{B}_{1},\dots,\widetilde{B}_{s})$. So  $e^{A_{k}}e^{B_{k}}=e^{B_{k}}e^{A_{k}}$ and
  $e^{\widetilde{A}_{l}}e^{\widetilde{B}_{l}}=e^{\widetilde{B}_{l}}e^{\widetilde{A}_{l}}$ , $k=1,\dots, r$, $l=1,\dots,s$.
By (ii), $A_{k}B_{k}=B_{k}A_{k}$, $k=1,\dots,r$. By Lemma ~\ref{L:1} for every $l=1,\dots,s$, there exists $Q_{l}\in
GL(2m_{l},\mathbb{C})$ such that $A^{\prime}_{l}=Q_{l}^{-1}\widetilde{A}_{l}Q_{l}=\mathrm{diag}(A^{\prime}_{l,1},
\overline{A^{\prime}_{l,1}})$ and $B^{\prime}_{l}=Q_{l}^{-1}\widetilde{B}_{l}Q_{l}=\mathrm{diag}(B^{\prime}_{l,1},
 \overline{B^{\prime}_{l,1}})$, where $A^{\prime}_{1}, B^{\prime}_{1}\in \mathbb{T}_{m_{l}}(\mathbb{C})$.  We see that
  $e^{A^{\prime}_{l}}=Q_{l}^{-1}e^{\widetilde{A}_{l}}Q_{l}$ and  $e^{B^{\prime}_{l}}=Q_{l}^{-1}e^{\widetilde{B}_{l}}Q_{l}$.
  Since $e^{A^{\prime}_{l}}e^{B^{\prime}_{l}}=e^{B^{\prime}_{l}}e^{A^{\prime}_{l}}$, we have
$e^{A^{\prime}_{l,1}}e^{B^{\prime}_{l,1}}=e^{B^{\prime}_{l,1}}e^{A^{\prime}_{l,1}}$. By  (ii),
 $A^{\prime}_{l,1}B^{\prime}_{l,1}=B^{\prime}_{l,1}A^{\prime}_{l,1}$,
 so $A^{\prime}_{l}B^{\prime}_{l}=B^{\prime}_{l}A^{\prime}_{l}$,
  and so $\widetilde{A}_{l}\widetilde{B}_{l}=\widetilde{B}_{l}\widetilde{A}_{l}$, $k=1,\dots,s$. We conclude that
   $AB=BA$. \qed
\\
   \bigskip

\textit{Proof of Lemma ~\ref{L:3}}. Let $M\in \mathbb{T}_{n}(\mathbb{C})$. Since $M$ is nilpotent of order $n$, so for all $t\in
\mathbb{R}$, one has $e^{tM}=\underset{k=0}{\overset{n}{\sum}}\frac{t^{k}M^{k}}{k!}.$ \  For all $1\leq i, j \leq n$, the function $P_{i,j}: \mathbb{R}
 \longrightarrow \mathbb{C}$, defined by  $P_{i,j}(t)=<
(e^{tM}-I_{n})e_{i},\ e_{j}>$  is a polynomial of
degree  $\leq 2n$. We have $P_{i,j}(m)=0$, for every $m\geq n$.
 Therefore, for all $1\leq i, j \leq n$ and all
$m\in\mathbb{N}$, $P_{i,j}(m)=0$,  thus for all $1\leq i, j \leq
n$, $P_{i,j}=0$. It follows that for all $t\in \mathbb{R}$, $e^{tM}=I_{n}$. By derivation, one has $Me^{tM}=0$,
 for all $t\in \mathbb{R}$, in particular, for $t=0$, $M=0$. This completes the proof. \qed
\bigskip

\vskip 0,4 cm

\bibliographystyle{amsplain}
\vskip 0,4 cm

\end{document}